\documentclass{amsart}

\usepackage{amsfonts,amsmath,amssymb}
\usepackage{latexsym,graphicx}
\newcommand{\bydef}{:=}

\newcommand{\bg}{{\overline{g}}}
\newcommand{\bt}{{\overline{t}}}

\newcommand{\bG}{{\overline{G}}}

\newcommand{\tA}{{\widetilde{A}}}

\newcommand{\bT}{{\overline{T}}}

\newcommand{\wt}[1]{\widetilde{#1}}

\newcommand{\vp}{\varphi}

\newcommand{\matr}[1]{\left(\begin{smallmatrix}#1\end{smallmatrix}\right)}

\DeclareMathOperator*{\ot}{\otimes}%allows placement of subscript below in displaymath
\newcommand{\beq}{\begin{equation}}
\newcommand{\eeq}{\end{equation}}
\newcommand{\beas}{\begin{eqnarray*}}
\newcommand{\eeas}{\end{eqnarray*}}
\newcommand{\id}{\mathrm{id}}%identity map

\newcommand{\cA}{\mathcal{A}}%algebras
\newcommand{\cB}{{\mathcal B}}
\newcommand{\cC}{\mathcal{C}}
\newcommand{\cD}{\mathcal{D}}

\newcommand{\cJ}{\mathcal{J}}
\newcommand{\cK}{\mathcal{K}}
\newcommand{\cL}{\mathcal{L}}

\newcommand{\cO}{\mathcal{O}}

\newcommand{\cQ}{\mathcal{Q}}
\newcommand{\cR}{\mathcal{R}}

\newcommand{\cV}{\mathcal{V}}

\newcommand{\frM}{\mathfrak{M}}
\newcommand{\frN}{\mathfrak{N}}

\newcommand{\R}{\mathbb{R}}
\newcommand{\C}{\mathbb{C}}

\newcommand{\NN}{\mathbb{N}}
\newcommand{\ZZ}{\mathbb{Z}}

\newcommand{\RR}{\mathbb{R}}

\newcommand{\CC}{\mathbb{C}}

\newcommand{\HH}{\mathbb{H}}
\newcommand{\OO}{\mathbb{O}}

\newcommand{\FF}{\mathbb{F}}

\newcommand{\LL}{\mathbb{L}}
\newcommand{\chr}[1]{\mathrm{char}\,#1}

\DeclareMathOperator{\Hom}{\mathrm{Hom}}
\DeclareMathOperator{\Gal}{\mathrm{Gal}}
\DeclareMathOperator{\End}{\mathrm{End}}
\DeclareMathOperator{\Alg}{\mathrm{Alg}}

\DeclareMathOperator{\im}{\mathrm{im}\,}

\DeclareMathOperator{\Aut}{\mathrm{Aut}}%automorphism group
\DeclareMathOperator{\AAut}{\mathbf{Aut}}%automorphism group scheme
\DeclareMathOperator{\supp}{\mathrm{Supp}}

\DeclareMathOperator{\lpc}{\mathit{L}_\pi^\chi}

\newcommand{\GL}{\mathrm{GL}}
\newcommand{\GLs}{\mathbf{GL}} %scheme

\newtheorem{theorem}{Theorem}[section]

\newtheorem{lemma}[theorem]{Lemma}
\newtheorem{corollary}[theorem]{Corollary}
\newtheorem{proposition}[theorem]{Proposition}

\newtheorem{example}[theorem]{Example}
\newtheorem{remark}[theorem]{Remark}

\newcommand{\CD}{\mathfrak{CD}} %Cayley-Dickson

\begin{document}

\title[On nonassociative graded-simple algebras]{On nonassociative graded-simple algebras\\ over the field of real numbers}

\author[Y.~Bahturin]{Yuri Bahturin}
\address{Department of Mathematics and Statistics, Memorial
University of Newfoundland, St. John's, NL, A1C5S7, Canada}
\email{bahturin@mun.ca}

\author[M.~Kochetov]{Mikhail Kochetov}
\address{Department of Mathematics and Statistics, Memorial
University of Newfoundland, St. John's, NL, A1C5S7, Canada}
\email{mikhail@mun.ca}

\thanks{The authors acknowledge support by the Natural Sciences and Engineering Research Council (NSERC) of Canada (DG \# 227060-2014 and 341792-2013).}

\begin{abstract}
We extend the loop algebra construction for algebras graded by abelian groups to study graded-simple algebras over the field of real numbers 
(or any real closed field). 
As an application, we classify up to isomorphism the graded-simple alternative (nonassociative) algebras and 
graded-simple finite-dimensional Jordan algebras of degree $2$.
We also classify the graded-division alternative (nonassociative) algebras up to equivalence. 
\end{abstract}

\maketitle

\section{Introduction}\label{sI}

Let $\cA$ be an algebra over a field $\FF$ and let $G$ be a group. A \emph{$G$-grading} on $\cA$ is a 
vector space decomposition $\Gamma:\;\cA=\bigoplus_{g\in G}\cA_g$ such that $\cA_g \cA_h\subset \cA_{gh}$ 
for all $g,h\in G$. The subspaces $\cA_g$ are called the \emph{homogeneous components}. 
A \emph{$G$-graded algebra} is an algebra with a fixed $G$-grading.
The nonzero elements $x\in\cA_g$ are said to be \emph{homogeneous of degree} $g$, which can be written as $\deg x=g$, 
and the \emph{support} of $\Gamma$ (or of $\cA$) is the set $\supp\Gamma\bydef\{g\in G\;|\;\cA_g\neq 0\}$. 
A subspace $\cB$ in $\cA$ (in particular, a subalgebra or ideal) is called 
\emph{graded} if $\cB=\bigoplus_{g\in G}\cB_g$ where $\cB_g\bydef \cB\cap \cA_g$.

Group gradings have been extensively studied for associative, Lie, Jordan, alternative, and other types of algebras. 
In particular, a classification of gradings is known for many simple finite-dimensional algebras of these types. 
The situation is especially well understood if the grading group $G$ is abelian and the ground field $\FF$ 
is algebraically closed (see e.g. a recent monograph \cite{EKmon} and the references therein). 

In the context of $G$-graded algebras, it is natural to consider graded analogs of the standard concepts.
For example, a \emph{homomorphism of $G$-graded algebras} $\psi:\cA\to \cB$ is a homomorphism of algebras 
that preserves degrees: $\psi(\cA_g)\subset \cB_g$ for all $g\in G$. 
In particular, $\cA$ and $\cB$ are \emph{graded-isomorphic} if there exists an isomorphism of graded algebras $\cA\to \cB$. 
A $G$-graded algebra $\cA$ is called \emph{graded-simple}, or \emph{simple as a graded algebra}, 
if $\cA^2\ne 0$ and $\cA$ has no graded ideals different from $0$ and $\cA$. 
Among these, we are especially interested in studying the \emph{graded-division} algebras, which are defined by the property 
that every nonzero homogeneous element is invertible (in an appropriate sense). 
A typical example of a graded-division algebra is the group algebra $\FF G$, which is naturally $G$-graded. 
More generally, the twisted group algebra $\FF^\gamma G$ (where $\gamma:G\times G\to\FF^\times$ is a $2$-cocycle)
is a graded-division associative algebra.
A graded algebra $\cA$ is said to be \emph{graded-central} if the only degree-preserving elements of its centroid 
are the scalar operators. We recall that the \emph{centroid} $C(\cA)$ of an algebra $\cA$ consists of all elements of 
$\End_\FF(\cA)$ that commute with the operators of left and right 
multiplication by the elements of $\cA$. If $\cA$ is unital, we can identify $C(\cA)$ with the center of $\cA$
(i.e., the set of elements that commute and associate with all elements of $\cA$).

For abelian $G$, a bridge between simple and graded-simple algebras is given by the \emph{loop algebra} 
construction in \cite{ABFP} (which appeared in \cite{BSZ01} in a special case). 
This construction generalizes the so-called multiloop algebras in Lie theory and, 
if $\FF$ is algebraically closed, yields a classification of graded-central-simple (i.e., graded-central and graded-simple) 
algebras in a given class provided that a classification of gradings is known for central simple algebras in this class 
(see Section~\ref{s:LC} for details). 
In this paper, we will assume all grading groups abelian, but $\FF$ will be the field of real numbers $\RR$. 
We will extend the loop construction to this setting. We do not use the topology of $\RR$, so our results are valid over 
any real closed field $\FF$ (with $\FF[\sqrt{-1}]$ playing the role of $\CC$).
 
Graded-division associative algebras that are finite-dimensional and graded-cent\-ral over $\RR$ have been classified 
in \cite{BZGRDA} up to equivalence (see Section~\ref{s:back}). That work was preceded by \cite{BZGDS}, 
where the \emph{division gradings} on finite-dimensional simple real associative algebras 
(i.e., the gradings that turn them into graded-division algebras) were classified up to equivalence, 
and \cite{ARE}, where these gradings were classified up to isomorphism as well as up to equivalence. 

In this paper, we will focus on alternative and Jordan algebras. For the background on these, we refer the reader to 
\cite{ZSSS} and \cite{J}.

Recall that an algebra is said to be \emph{alternative} if the associator $(x,y,z)\bydef (xy)z-x(yz)$ is an alternating 
function of the variables $x,y,z$ or, equivalently, the left and right alternative identities hold: $x(xy)=x^2y$ and $(yx)x=yx^2$.
By Artin's Theorem, this implies that any subalgebra generated by two elements is associative.
The definition of inverses in a unital alternative algebra is the same as in the associative case, 
and the set of invertible elements is closed under multiplication (it is a so-called \emph{Moufang loop}).
Kleinfeld proved that any simple alternative algebra is either associative or an octonion algebra over a field.

Group gradings on octonion algebras over any field were described in \cite{E} (see also \cite{EKmon}). 
(Here, as in the case of simple Lie algebras, there is no loss of generality in assuming the grading group abelian 
because the elements of the support always commute.) 
A classification of gradings up to isomorphism is given in \cite{EK_G2D4}. 
We will use our extension of the loop algebra construction to derive a classification up to isomorphism 
of graded-simple alternative algebras that are graded-central over $\RR$ and not associative (Theorem~\ref{th:alternative_iso}).
Among these, we will also classify the graded-division algebras up to equivalence (Theorem~\ref{th:alternative_equ} and Corollary~\ref{cor:alternative_equ}).

Let $\FF$ be a field of characteristic different from $2$. A \emph{Jordan algebra} over $\FF$ is a commutative algebra in which 
the identity $(xy)x^2=x(yx^2)$ holds or, equivalently, the operators $L_x$ and $L_{x^2}$ commute for any element $x$, 
where $L_x(y)\bydef xy$. Any associative algebra $\cA$ gives rise to a Jordan algebra, denoted $\cA^{(+)}$, which 
coincides with $\cA$ as a vector space, but whose multiplication is given by $x\circ y\bydef\frac12(xy+yx)$.
The Jordan algebras that can be embedded into $\cA^{(+)}$ for some associative algebra $\cA$ are called \emph{special}.
Zelmanov proved that any simple Jordan algebra is either special or an Albert algebra (which has dimension $27$ over its center). 

Abelian group gradings on finite-dimensional simple special Jordan algebras 
over an algebraically closed field of characteristic $0$ were described in \cite{BS,BSZ05} and 
classified up to isomorphism and up to equivalence in \cite{EKmon} for any characteristic different from $2$. 
The finite-dimensional central simple special Jordan algebras of degree $\ge 3$ over $\RR$ can be dealt with following the approach of
\cite{BKR_L}, where a classification of gradings for classical central simple Lie algebras over $\RR$ was obtained by
transferring the problem to associative algebras with involution. Here we will consider degree $2$, i.e., Jordan algebras of bilinear forms.
We will apply our extension of the loop construction to classify the corresponding graded-simple algebras up to isomorphism (Theorem~\ref{th:J_deg2_iso}),
and characterize those that are graded-division algebras (Corollary~\ref{cor:J_deg2_div}).

It should be noted that abelian group gradings on the Albert algebra over $\CC$ were classified in \cite{DM}
and over any algebraically closed field of characteristic different from $2$ in \cite{EK1} 
(see also \cite{EKmon}). In particular, it has a division grading by the group $\ZZ_3^3$ (which appeared in \cite{Griess}) if the characteristic 
is not $3$. However, none of the three real forms of the complex Albert algebra admits a division grading, as follows from 
the classification of fine gradings in \cite{CDM} (since the identity component always contains nontrivial idempotents). 

The paper is structured as follows. In Section~\ref{s:back}, we review some terminology and facts about gradings. Section~\ref{s:LC} is devoted 
to the loop construction over $\RR$: Theorem~\ref{th:loop1} establishes the existence of a ``loop model'' for a graded-simple algebra 
and Corollary~\ref{cor:loop} answers the question to what extent it is unique. 
(The answer turns out to be the same as for the case of ``split centroid'' in \cite{ABFP}.)
In fact, we also give a description of isomorphisms between graded-simple algebras, which can be conveniently expressed as equivalence of certain groupoids
(Theorem~\ref{th:loop2}). Sections~\ref{s:alternative} and \ref{s:Jordan} deal with our applications of the loop construction to alternative and Jordan 
algebras, respectively.

Throughout the paper, an unadorned symbol $\ot$ denotes the tensor product over the ground field, which is usually $\RR$. 
We will also use the following notation for abelian groups: for any $m\in\NN$, we have the group homomorphism $[m]:G\to G$ sending 
$g\mapsto g^m$, and the associated subgroups of $G$: $G_{[m]}=\ker[m]$ and $G^{[m]}=\im[m]$. The order of an element 
$g\in G$ will be denoted by $o(g)$.

\section{Background on gradings}\label{s:back}

\subsection{Gradings induced by a group homomorphism}\label{ss:induced_grading}

Given a $G$-grading $\Gamma:\;\cA=\bigoplus_{g\in G}\cA_g$ and a group homomorphism $\alpha:G\to G'$, 
we can define a $G'$-grading on $\cA$ as follows: for any $g'\in G'$, set 
$\cA'_{g'}\bydef\bigoplus_{g\in\alpha^{-1}(g')}\cA_g$.
Then $\cA=\bigoplus_{g'\in G'}\cA'_{g'}$ is a $G'$-grading, which will be denoted by ${}^\alpha\Gamma$. 
Thus, if $\alpha$ is fixed, every $G$-graded algebra $\cA$ can be considered as a $G'$-graded algebra, which 
will be denoted by ${}^\alpha\cA$ or just $\cA$ if there is no risk of confusion. 
In particular, if $G$ is a subgroup of $G'$ then every $G$-graded algebra is also $G'$-graded, 
with the same support and homogeneous components. 
More generally, if $\alpha$ is injective on the support then ${}^\alpha\Gamma$ has the same components as $\Gamma$, 
but relabeled through $\alpha$. Otherwise, ${}^\alpha\Gamma$ is a so-called \emph{proper coarsening} of $\Gamma$.

If $\psi:\cA\to \cB$ is a homomorphism of $G$-graded algebras and we consider $\cA$ and $\cB$ as $G'$-graded 
through $\alpha$ as above, then $\psi$ is also a homomorphism of $G'$-graded algebras. Thus, $\alpha$ determines 
a functor $F_\alpha$ from the category of $G$-graded algebras to the category of $G'$-graded algebras: 
$F_\alpha(\cA)={}^\alpha\cA$ and $F_\alpha(\psi)=\psi$.

\subsection{Weak isomorphism and equivalence}\label{ss:wiso_equiv}

We call a $G$-graded algebra $\cA$ and a $G'$-graded algebra $\cB$ \emph{graded-weakly-isomorphic} if 
${}^\alpha\cA$ is graded-isomorphic to $\cB$ for some group isomorphism $\alpha:G\to G'$, 
i.e., there exists an isomorphism of algebras $\psi:\cA\to\cB$ such that $\psi(\cA_g)=\cB_{\alpha(g)}$ for all $g\in G$. 

There is a more general kind of relabeling: we say that a $G$-graded algebra $\cA$ and a $G'$-graded algebra $\cB$ are 
\emph{graded-equivalent} if there exists an isomorphism of algebras $\psi:\cA\to\cB$ such that, for each $g\in G$, 
there is $g'\in G'$ such that $\psi(\cA_g)=\cB_{g'}$. 

The support of a division grading on an alternative (in particular, associative) algebra is a subgroup. 
If, for two such graded-division algebras, we replace the grading groups by the supports, then equivalence is the same as weak isomorphism.

\subsection{Graded modules}

Let $\cA$ be a $G$-graded associative algebra over a field $\FF$. A $G$-graded vector space $\cV$ over $\FF$, i.e., a vector space 
with a fixed decomposition $\cV=\bigoplus_{g\in G}\cV_g$, is a \emph{graded left $\cA$-module} if it is equipped 
with a left $\cA$-action such that $\cA_g\cV_h\subset\cV_{gh}$ for all $g,h\in G$.
The concepts of \emph{homogeneous element}, \emph{support}, \emph{isomorphism}, etc., are defined in the same way as for graded algebras.

If $\cA$ is a graded-division algebra then every graded $\cA$-module is free. More precisely, it admits a basis consisting of 
homogeneous elements, the $\FF$-span $V$ of such a basis is a $G$-graded vector space, and $\cV\cong\cA\otimes V$ as graded $\cA$-modules,
where the $G$-grading on the tensor product of $G$-graded vector spaces $U$ and $V$ is defined by setting $\deg(u\otimes v)\bydef\deg(u)\deg(v)$.

\section{Loop construction}\label{s:LC} 

\subsection{Centroid of a graded algebra}

Given a $G$-graded algebra $\cB$ over a field $\FF$, a linear map $f:\cB\to \cB$ is said to be 
homogeneous of degree $g\in G$ if $f(\cB_h)\subset \cB_{gh}$ for all $h\in G$. 
In particular, the set of all elements of the centroid $C(\cB)$ that satisfy this condition will be denoted by $C(\cB)_g$. 

If $\psi:\cB\to\cB'$ is an isomorphism of $G$-graded algebras then we have an isomorphism of the centroids 
$C(\psi):C(\cB)\to C(\cB')$ defined by $C(\psi)(c)=\psi\circ c\circ\psi^{-1}$ for all $c\in C(\cB)$. 
(If the algebras are unital and we identify the centroids with the centers then $C(\psi)$ is just a restriction of $\psi$.)
It is clear that $C(\psi)$ maps $C(\cB)_g$ onto $C(\cB')_g$ for any $g\in G$.

Suppose $\cB$ is graded-simple. Then, by \cite[Proposition 2.16]{BN}, we have $C(\cB)=\bigoplus_{g\in G}C(\cB)_g$, 
$C(\cB)$ is a commutative associative graded-division algebra (a \emph{graded-field}), and $\cB$ is a graded $C(\cB)$-module (hence free).   
Moreover, the graded-simple algebra $\cB$ is simple if and only if $C(\cB)$ is a field \cite[Lemma 4.2.2]{ABFP}. 
In any case, the identity component $C(\cB)_e$ is a field containing $\FF$, and $\cB$ is a graded-central algebra 
over $C(\cB)_e$, with exactly the same homogeneous components, but now viewed as $C(\cB)_e$-subspaces. 
For this reason, it is natural to restrict ourselves to the graded-central case.

So, suppose $\cB$ is a graded-central-simple algebra over $\FF$. Then the homogeneous components of $C(\cB)$ are 
one-dimensional: $C(\cB)_h=\FF u_h$ for every $h\in H$ where $H$ is the support of $C(\cB)$, which is a subgroup of $G$, 
and $u_h$ is an invertible element. 
Thus, $C(\cB)$ is a twisted group algebra of $H$, with its natural $H$-grading considered as a $G$-grading. 
Following \cite{ABFP}, we will say that the centroid is \emph{split} if $C(\cB)$ is graded-isomorphic 
to the group algebra $\FF H$. 
This is always the case if $\FF$ is algebraically closed \cite[Lemma 4.3.8]{ABFP}, 
but not if $\FF$ is the field of real numbers. For example, $C(\cB)$ can be the field $\CC$ with a nontrivial grading: 
$\CC=\CC_e\oplus\CC_h$ where $\CC_e=\RR$, $\CC_h=\RR i$, and $h$ is an element in $G$ of order $2$. 
The loop construction in \cite{ABFP}, which we are now going to review, produces graded-central-simple algebras 
with split centroid. We will extend this construction to cover all possible centroids in the case $\FF=\RR$ 
(or any real closed field).

\subsection{Loop algebras with a split centroid}

Let $\pi:G\to\bG$ be an epimorphism of abelian groups and let $H$ be the kernel of $\pi$. 
As discussed in Subsection \ref{ss:induced_grading}, any $G$-graded algebra $\cB$ can be regarded as a $\bG$-graded algebra,
by setting $\cB_\bg=\bigoplus_{g\in\pi^{-1}(\bg)}\cB_g$ for any $\bg\in\bG$, and this gives us a functor 
from the category of $G$-graded algebras to the category of $\bG$-graded algebras. 
The loop construction is the right adjoint of this functor (see \cite[Remark 3.3]{EK_loop}), defined as follows. 
For a given $\bG$-graded algebra $\cA$, consider the tensor product $\cA\ot\FF G$.  
The loop algebra $L_\pi(\cA)$ is the following subalgebra of $\cA\otimes\FF G$:
\begin{equation*}
L_\pi(\cA)=\bigoplus_{g\in G} \cA_{\pi(g)}\ot g, 
\end{equation*}
which is naturally $G$-graded: $L_\pi(\cA)_g=\cA_{\pi(g)}\ot g$.
If $\psi:\cA\to\cA'$ is a homomorphism of $\bG$-graded algebras then the linear map $L_\pi(\psi):L_\pi(\cA)\to L_\pi(\cA')$ 
that sends $a\otimes g\mapsto\psi(a)\otimes g$, for all $a\in\cA_{\pi(g)}$ and $g\in G$, is a homomorphism of $G$-graded 
algebras.

Note that $L_\pi(\cA)$ is unital if and only if so is $\cA$. Also, if $\mathfrak{V}$ is a variety of algebras, i.e., 
a class defined by polynomial identities, and $\FF$ is infinite, then $L_\pi(\cA)$ belongs to $\mathfrak{V}$ 
if and only if so does $\cA$. Indeed, if $\cA\in\mathfrak{V}$ then $\cA\ot\FF G\in\mathfrak{V}$ 
(since $\FF G$ is commutative, see e.g. \cite[p.~10]{GZ}) and hence $L_\pi(\cA)\in\mathfrak{V}$; 
the converse follows from the fact that $\cA$ is a quotient of $L_\pi(\cA)$ (by means of $a\otimes g\mapsto a$).

One of the assertions of the ``Correspondence Theorem'' (see \cite[Theorem 7.1.1]{ABFP}) is that any $G$-graded algebra 
$\cB$ that is graded-central-simple and has a split centroid is graded-isomorphic to $L_\pi(\cA)$ 
for some central simple algebra $\cA$ equipped with a $\bG$-grading where $\bG=G/H$, $H$ is the support of $C(\cB)$, 
and $\pi:G\to\bG$ is the natural homomorphism. Note that the condition on the centroid of $\cB$ is necessary because,
for any central simple algebra $\cA$ with a $\bG$-grading, $L_\pi(\cA)$ is a graded-central-simple algebra whose 
centroid can be identified with $L_\pi(\FF)=\FF H$, where $h\in H$ acts on $L_\pi(\cA)$ as follows:
\[
h(a\otimes g)=a\otimes hg\text{ for all } a\in\cA_{\pi(g)},\: g\in G.
\]
We will generalize the loop construction to produce $G$-graded algebras with nonsplit centroids over $\RR$. 

Another assertion of the ``Correspondence Theorem'' tells us that the $G$-graded algebra $\cB$ determines the 
$\bG$-graded algebra $\cA$ up to isomorphism and twist in the following sense. 
Fix a transversal for the subgroup $H$ in $G$, so we have a section 
$\xi:\bG\to G$ 
(which is not necessarily a group homomorphism). This defines a $2$-cocycle on $\bG$ with values in $H$, namely, 
\[
\sigma(\bg_1,\bg_2)=\xi(\bg_1)\xi(\bg_2)\xi(\bg_1\bg_2)^{-1}.
\]
Now, given a character $\lambda:H\to\FF^\times$ and a $\bG$-graded algebra $\cA$, we define $\cA^\lambda$ to be 
$\cA$ as a $\bG$-graded space, but with a twisted multiplication:
\begin{equation}\label{eq:twisted_multiplication}
a_1\ast a_2=\lambda(\sigma(\bg_1,\bg_2))a_1 a_2\text{ for all } a_1\in \cA_{\bg_1},\: a_2\in \cA_{\bg_2}.
\end{equation}
In the above formula, $\lambda\circ\sigma$ is a $2$-cocycle on $\bG$ with values in $\FF^\times$, so this is actually a
standard cocycle twist of a graded algebra. Note that the isomorhism class of $\cA^\lambda$ does not depend on the choice of 
the transversal. Also, if $\lambda$ extends to a character $G\to\FF^\times$ 
(for example, if $\FF$ is algebraically closed or if $H$ is a direct summand of $G$) 
then $\cA^\lambda$ is graded-isomorphic to $\cA$.

\begin{remark}\label{rem:mod_over_centroid}
As a graded $L_\pi(\FF)$-module, $L_\pi(\cA)\cong L_\pi(\FF)\otimes\underline{\cA}$ where $\underline{\cA}$ is 
$\cA$ regarded as a $G$-graded vector space with $\deg a=\xi(\bg)$ for all nonzero $a\in\cA_\bg$, $\bg\in\bG$. 
An isomorphism $L_\pi(\FF)\otimes\underline{\cA}\to L_\pi(\cA)$ is given by $h\otimes a\mapsto a\otimes h\xi(\bg)$, 
where we have identified $L_\pi(\FF)$ with $\FF H$.
In particular, if $\xi$ is a group homomorphism (which can happen if and only if $H$ is a direct summand of $G$) 
then $\underline{\cA}={}^\xi\cA$ is a $G$-graded algebra, and $L_\pi(\cA)\cong \FF H\otimes {}^\xi\cA$ as graded algebras.
\end{remark}

We are going to generalize, over the field $\RR$, the above two assertions of the ``Correspondence Theorem'' to 
graded-central-simple algebras whose centroid is not necessarily split. The following result will be useful:

\begin{lemma}\label{lm:2} 
If $\FF$ is algebraically closed then every graded-automorphism of the centroid $L_\pi(\FF)=\FF H$  
is induced by a graded-automorphism of $L_\pi(\cA)$.
\end{lemma}

\begin{proof}
Any automorphism of $\FF H$ as a graded algebra has the form: $\psi_0(h)=\mu(h)h$, for any $h\in H$, 
where $\mu:H\to\FF^\times$ is a character. We extend $\mu$ to a character of $G$ and set $\psi(a\ot g)=\mu(g)a\ot g$ for  
all $g\in G$ and $a\in \cA_{\pi(g)}$. This is an automorphism of $L_\pi(\cA)$ and, for any $c\in\FF H$, we have  
$\psi(c(a\ot g))=\psi_0(c)(\psi(a\ot g))$, which means that $C(\psi)=\psi_0$.
\end{proof}

\subsection{Real loop algebras with a nonsplit centroid}\label{sRLANSC} 

From now on, assume that the ground field is $\RR$ (or, more generally, any real closed field).
Let $\pi:G\to\bG$ be an epimorphism of abelian groups, $H=\ker\pi$.
We will apply the loop construction from the previous subsection over the algebraic closure $\CC$ and 
then define a suitable Galois descent to $\RR$. Denote the generator of $\Gal(\CC/\RR)$ by $\iota$, i.e., 
$\iota(z)=\bar{z}$ for all $z\in\CC$.

Let $\chi:G\to\CC^\times$ be a character with values in the unit circle $U\subset\CC$ 
(which is automatic for finite groups). Then the $\RR$-linear operator on $\CC G=\CC\otimes\RR G$ defined by 
$z\otimes g\mapsto \bar{z}\chi(g)\otimes g$ is involutive, so we obtain a semilinear action of $\Gal(\CC/\RR)$ on $\CC G$ 
and hence a real form of $\CC G$ as a $G$-graded algebra. (In fact, any such real form is obtained in this way.) 
The homogeneous component of degree $g$ is $\RR u_g$, where $u_g=z_g\otimes g$ and $z_g\in\CC^\times$ is any 
element satisfying $z_g/\bar{z}_g=\chi(g)$. We have 
\[
u_{g_1}u_{g_2}=\gamma(g_1,g_2)u_{g_1 g_2}\text{ for all }g_1,g_2\in G,
\]
so our real form is the twisted group algebra $\cD=\RR^\gamma G$ where the $2$-cocycle $\gamma:G\times G\to\RR^\times$ 
is given by $\gamma(g_1,g_2)=z_{g_1}z_{g_2}/z_{g_1g_2}$. We will choose $z_g$ to be in the unit circle, i.e.,
a square root of $\chi(g)$. Then $\gamma$ takes values in $\{\pm 1\}$.

Now let $\cA$ be a $\bG$-graded algebra over $\RR$. Similarly to the above, $\cA\ot\CC\ot\RR G$ acquires a semilinear action
of $\Gal(\CC/\RR)$ defined by $\iota\cdot(a\ot z\ot g)=a\ot\bar{z}\chi(g)\ot g$, which restricts to a degree-preserving 
action on $L_\pi(\cA\ot\CC)=\bigoplus_{g\in G}\cA_{\pi(g)}\ot\CC\ot\RR G$. 
As a result, we obtain the following $G$-graded $\RR$-form of $L_{\pi}(\cA\ot\CC)$:
\[
L_\pi^\chi(\cA)\bydef\{ x\in L_{\pi}(\cA\ot\CC) \mid \iota\cdot x=x\}=\bigoplus_{g\in G} \cA_{\pi(g)}\ot u_g
\subset \cA\ot\RR^\gamma G.
\]
By construction, $L^\chi_\pi(\cA)\ot\CC\cong L_\pi(\cA\ot\CC)$. It follows that $L^\chi_\pi(\cA)$ is unital 
if and only if so is $\cA$, and $L^\chi_\pi(\cA)$ belongs to a variety $\mathfrak{V}$ if and only if so does $\cA$. 
Indeed, we already know that $L_\pi$ preserves these properties, and they are not affected by field extensions.
By the same argument, if $\cA$ is central simple then $L^\chi_\pi(\cA)$ is graded-central-simple. 

The action of $\Gal(\CC/\RR)$ passes on to the centroid of $L_\pi(\cA\otimes\CC)$: 
$(\iota\cdot c)(x)=\iota\cdot c(\iota^{-1}\cdot x)$, for any $x\in L_\pi(\cA\ot\CC)$ and $c$ in the centroid. 
Assuming that $\cA$ is central simple, we have identified the centroid with $\CC H=\CC\ot\RR H$. 
With this identification, we have $\iota\cdot(z\ot h)=\bar{z}\chi(h)\ot h$ for all $z\in\C$ and $h\in H$. 
Therefore, the centroid of $L^\chi_\pi(\cA)$ can be identified with $\cD_H\bydef\bigoplus_{h\in H}\cD_h$, 
which is a graded subalgebra of the real form $\cD$ of $\CC G$ defined above. Note that $\cD_H=L_\pi^\chi(\RR)$. 

We are now ready to extend a part of the ``Correspondence Theorem'' to the case of nonsplit centroid.

\begin{theorem}\label{th:loop1}
Let $\cB$ be a $G$-graded algebra that is graded-central-simple over $\RR$. 
Let $H$ be the support of the centroid $C(\cB)$  and $\pi:G\to\bG=G/H$ be the natural homomorphism. 
Then, for some character $\chi:G\to U\subset\CC^\times$, there exists a central simple algebra $\cA$ over $\R$ and a
$\bG$-grading on $\cA$ such that $\cB\cong L^\chi_\pi(\cA)$ as $G$-graded algebras. Moreover, we can take
any character satisfying $L^\chi_\pi(\RR)\cong C(\cB)$.
\end{theorem}

\begin{proof}
Since $\CC$ is algebraically closed, the centroid of $\cB\ot\CC$ is split, hence, by the ``Correspondence Theorem'' 
from \cite{ABFP}, there exists a central simple complex algebra $\wt{\cA}$ with a $\bG$-grading such that 
$\cB\ot\CC\cong L_\pi(\wt{\cA})$. Hence, $\cB$ is isomorphic to an $\RR$-form of $L_\pi(\wt{\cA})$, and we need to 
determine all semilinear $\Gal(\CC/\RR)$-actions on the latter (up to isomorphism) to recover $\cB$ by Galois decent. 

First consider the induced $\Gal(\CC/\RR)$-action on the centroid $\CC H$ of $L_\pi(\wt{\cA})$.
As already observed, any automorphism of $\CC H$ as a graded $\CC$-algebra is of the form $zh\mapsto z\chi(h)h$, 
where $\chi:H\to\C^\times$ is a character. It follows that any degree-preserving $\CC$-antilinear automorphism  
has the form $zh\mapsto\bar{z}\chi(h)h$. In particular, we have $\iota\cdot(zh)=\bar{z}\chi(h)h$ for some $\chi$,
which must take values in the unit circle $U$ because $\iota^2=\id$. Since $U$ is a divisible abelian group, 
we can extend $\chi$ to the whole of $G$, with all values still in $U$. 

Since the action of $\Gal(\CC/\RR)$ on $L_\pi(\wt{\cA})=\bigoplus_{g\in G} \wt{\cA}_{\pi(g)}\ot g$ is degree-preserving, 
we can write, for any $g\in G$, $\iota\cdot(a\ot g)=\vp_g(a)\chi(g)\ot g$ for all $a\in \wt{\cA}_{\pi(g)}$. 
It is clear that $\vp_g:\wt{\cA}_{\pi(g)}\to \wt{\cA}_{\pi(g)}$ is a $\CC$-antilinear map. 
Also, $\chi(g)\overline{\chi(g)}=1$ implies $\vp_g^2=\id$. 
We claim that $\vp_g$ depends only on $\pi(g)$. Indeed, for all $a\in\wt{\cA}_{\pi(g)}$ and $h\in H$, we can consider $h$ 
as an element of the centroid and compute:
\[
\begin{split}
\vp_{gh}(a)\chi(gh)\ot hg&=\iota\cdot(a\ot hg)=\iota\cdot\big(h(a\ot g)\big)=(\iota\cdot h)\big(\iota\cdot(a\ot g)\big)\\
&=\chi(h)h\big(\vp_g(a)\chi(g)\ot g\big)=\chi(h)\chi(g)\vp_g(a)\ot hg.
\end{split}
\]
Thus, $\vp_{hg}=\vp_g$, which proves the claim. 

Define $\vp:\wt{\cA}\to\wt{\cA}$ by setting $\vp|_{\wt{\cA}_{\pi(g)}}=\vp_g$, for any $g\in G$. 
Clearly, $\vp$ is $\CC$-antilinear and $\vp^2=\id$. Moreover, $\vp$ is an isomorphism:
\[
\begin{split}
\big(\vp_g(a)\chi(g)\ot g\big)\big(\vp_{g'}(a')\chi(g')\ot g'\big)
&=\big(\iota\cdot(a\ot g)\big)\big(\iota\cdot(a'\ot g')\big)=\iota\cdot(aa'\ot gg')\\
&=\vp_{gg'}(aa')\chi(gg')\ot gg',
\end{split}
\]
for all $g,g'\in G, a,a'\in\wt{\cA}$, so $\vp_g(a)\vp_{g'}(a')=\vp_{gg'}(aa')$, as claimed.

Therefore, $\vp$ gives a Galois descent for $\wt{\cA}$. Let $\cA=\bigoplus_{\bg\in\bG}\cA_{\bg}$ where 
\[
\cA_{\bg}=\{ a\in \tA_{\bg}\mid\vp(a)=a\}.
\]
Then $\cA\ot\C\cong\wt{\cA}$ by means of the mapping $a\ot z\mapsto az$.
Finally, the algebra of fixed points of $\iota$ in $L_\pi(\wt{\cA})$ is the $\RR$-span of the set of all $az_g\ot g$ 
where $a\in \cA_{\pi(g)}$ and $z_g\in\C^\times$ satisfy $z_g/\bar{z}_g=\chi(g)$. 
This span is isomorphic to $L_\pi^\chi(\cA)$ under the isomorphism $L_\pi(\cA\ot\CC)\to L_\pi(\wt{\cA})$ 
sending $a\ot z\ot g$ to $az\ot g$.

For the last assertion, observe that, in view of Lemma \ref{lm:2}, the group of automorphisms of the graded $\CC$-algebra
$L_\pi(\wt{\cA})$ acts transitively on each isomorphism class of graded $\RR$-forms of the centroid $\CC H$. 
This allows us to adjust the $\Gal(\CC/\RR)$-action on $L_\pi(\wt{\cA})$ so that the induced action on the centroid 
is given by $\iota\cdot(zh)=\bar{z}\chi(h)h$ where $\chi$ is any character such that the corresponding graded $\RR$-form of 
$\CC H$ (spanned by the elements $z_h h$) is isomorphic to $C(\cB)$.
\end{proof}

The nonzero homogeneous elements of degree $g$ in $L_\pi^\chi(\cA)$ have the form $a\otimes u_g$ where $0\ne a\in\cA_{\pi(g)}$, so the support of 
$L_\pi^\chi(\cA)$ is the inverse image under $\pi$ of the support of $\cA$. Remark \ref{rem:mod_over_centroid} has the following analog:

\begin{remark}\label{rem:mod_over_centroid_nonsplit}
As a graded $L_\pi^\chi(\FF)$-module, $L_\pi^\chi(\cA)\cong L_\pi^\chi(\FF)\otimes\underline{\cA}$ where $\underline{\cA}$
is $\cA$ regarded as a $G$-graded vector space with $\deg a=\xi(\bg)$ for all nonzero $a\in\cA_\bg$, $\bg\in\bG$. 
An isomorphism $L_\pi^\chi(\FF)\otimes\underline{\cA}\to L_\pi^\chi(\cA)$ is given by 
$u_h\otimes a\mapsto a\otimes u_h u_{\xi(\bg)}$, where we have identified $L_\pi^\chi(\FF)$ with $\cD_H$. 
In particular, if $\xi$ is a group homomorphism then we can choose $\chi$ to be trivial on the complement $\xi(\bG)$ 
of $H$ in $G$, which entails $L_\pi^\chi(\cA)\cong L_\pi^\chi(\FF)\otimes {}^\xi\cA$ as graded algebras.
\end{remark}

Since the elements $u_g$ are invertible, we also observe the following:

\begin{remark}\label{rem:div}
If $\cA$ is alternative (in particular, associative) or Jordan, 
then $\cA$ is a graded-division algebra with respect to $\bG$ if and only if 
$L_\pi^\chi(\cA)$ is a graded-division algebra with respect to $G$. 
\end{remark}

\begin{example}\label{ex:1}
If $\cB=\bigoplus_{g\in G}\cB_g$ is a graded-field with $\cB_e=\RR 1$ then $\cB\cong L_\pi^\chi(\RR)$ 
where $H=G$ and $\pi:G\to \{ 1\}$. If $G$ is finite, we can write it as a direct product of cyclic groups: 
$G=\langle g_1\rangle\times\cdots\times\langle g_s\rangle$, and choose roots of unity $z_j$, $j=1,\ldots,s$, 
satisfying $z_j^2=\chi(g_j)$. Then 
\[
L_\pi^\chi(\RR)=\RR[u_1]\ot\cdots\ot\RR[u_s]\text{ where }\deg u_j=g_j
\text{ and }u_j^{o(g_j)}=z_j^{o(z_j)}\in\{\pm 1\}.
\]
This is the description of such algebras given in \cite{BZGRDA}.
\end{example}

\begin{example}\label{ex:2}
Let $\cC$ be a real octonion algebra, i.e., either the octonion division algebra $\OO$ or 
the split octonion algebra $\OO_s$. 
Suppose $\pi:G\to\bG$ is an epimorphism of abelian groups and $\cC$ is given a $\bG$-grading. 
Then, for any character $\chi:G\to U$, the $G$-graded algebra $L_\pi^\chi(\cC)$ is a real alternative nonassociative 
algebra that is graded-central-simple. Conversely, every such algebra is graded-isomorphic to an algebra 
of the form $L_\pi^\chi(\cC)$ where $\cC$ is a simple alternative nonassociative algebra with centroid $\RR$, 
which is a real octonion algebra by Kleinfeld's Theorem \cite{K1,K2}(see also \cite{ZSSS}). 
\end{example}

\subsection{Isomorphism problem for real loop algebras}

We have shown how to obtain any $G$-graded algebra $\cB$ that is graded-central-simple over $\RR$ from a suitable 
central simple algebra $\cA$ equipped with a grading by a quotient group of $G$: $\cB\cong L_\pi^\chi(\cA)$.
Our goal is to classify the graded-central-simple algebras up to isomorphism, so we are going to investigate to what extent 
$\cA$ is determined by $\cB$, but first we want to fix the parameters $\pi$ and $\chi$. Recall that $\pi$ is the natural 
homomorphism $G\to\bG$, where $\bG=G/H$ and $H$ is the support of the centroid $C(\cB)$, so $\pi$ is determined by $\cB$.
As to $\chi$, it is a homomorphism $G\to U\subset\CC^\times$ such that $L_\pi^\chi(\RR)\cong C(\cB)$. 
Let us see to what extent $\chi$ is determined. 

\begin{lemma}\label{lm:1}
$L_\pi^\chi(\RR)\cong L_\pi^{\chi'}(\RR)$ if and only if $\chi|_{H_{[2]}}=\chi'|_{H_{[2]}}$.
\end{lemma}

\begin{proof}
Recall that $L_\pi^\chi(\RR)$ and $L_\pi^{\chi'}(\RR)$ are the $\RR$-forms of $\CC H$ given by the $\Gal(\CC/\RR)$-actions 
coming from $\chi$ and $\chi'$. These $\RR$-forms are isomorphic if and only if there exists an automorphism of $\CC H$ 
(as a graded algebra) that sends one onto the other, which happens if and only if there exists a character 
$\mu: H\to\CC^\times$ such that $\overline{\mu(h)}\chi'(h)=\mu(h)\chi(h)$ for all $h\in H$. Normalizing $\mu$, we can assume 
that $\mu:H\to U$, so the condition becomes $\chi'|_H=\mu^2\chi|_H$, or $\chi'|_H(\chi|_H)^{-1}\in \Hom_\ZZ(H,U)^{[2]}$. 
This latter condition is equivalent to $\chi'|_{H_{[2]}}=\chi|_{H_{[2]}}$, because we have the exact sequence
\[
\Hom_\ZZ(H,U) \xrightarrow{[2]} \Hom_\ZZ(H,U)\xrightarrow{\mathrm{res}}\Hom_\ZZ(H_{[2]},U)\to 1 
\]
as the result of applying the functor $\Hom_\ZZ(\,\textbf{.}\,,U)$ to the exact sequence
\[
1\to H_{[2]}\to H \xrightarrow{[2]}  H.
\]
The functor is exact because $U$ is divisible.
\end{proof}

In order to classify graded-central-simple algebras up to isomorphism, we may first sort them according to the isomorphism
type of the centroid and then solve the isomorphism problem for two algebras with isomorphic centroids. The first step is 
achieved by Lemma \ref{lm:1}, since $\lpc(\RR)$ is the centroid of $\lpc(\cA)$. 
For the second, we will fix $\pi:G\to\bG$ and $\chi:G\to U$ and determine when 
$\lpc(\cA)$ and $\lpc(\cA')$ are isomorphic as graded algebras. To this end, it is convenient to
introduce the following notation. 

Let $\frM$ be the category whose objects are the central simple algebras over $\RR$ equipped with a $\bG$-grading and 
whose morphisms are the isomorphsms of $\bG$-graded algebras. 
Let $\frN$ be the category whose objects are the graded-simple $G$-graded algebras with centroid isomorphic to 
$\cL_\pi^\chi(\RR)$ and whose morphisms are the isomorphisms of $G$-graded algebras. 
Thus, $\frM$ and $\frN$ are \emph{groupoids} (categories in which all morphisms are invertible).

Then $\lpc$ is a functor $\frM\to\frN$. Indeed, we have already seen that, for any object $\cA$ of $\frM$, the loop 
algebra $\lpc(\cA)$ is an object of $\frN$. 
Now let $\psi: \cA\to \cA'$ be an isomorphism of $\bG$-graded algebras. Define $\lpc(\psi)$ as the restriction of 
$\psi\ot\id_{\R^\gamma G}$ to $\lpc(\cA)\subset \cA\ot\R^\gamma G$, i.e., $\lpc(\psi)$ sends 
$a\ot u_g\mapsto \psi(a)\ot u_g$ for all $a\in\cA_{\pi(g)}$ and $g\in G$. 
It is clear that $\lpc(\psi)$ is an isomorphism of $G$-graded algebras and that $\psi_1\ne\psi_2$ implies 
$\lpc(\psi_1)\ne\lpc(\psi_2)$, so $\lpc$ is a faithful functor $\frM\to\frN$. 
By Theorem~\ref{th:loop1}, it is essentially surjective. However, in general, it fails to be full. 
To obtain an equivalence of categories, we extend $\frM$ and $\lpc$ as follows (cf. \cite{EK_loop}). 

Fix a transversal for $H$ in $G$ and let $\xi:\bG\to G$ be the corresponding section of $\pi$, which determines
a $2$-cocycle $\sigma:\bG\times\bG\to H$. 
We define a groupoid $\widetilde{\frM}$ that has the same objects as $\frM$, but more morphisms. 
Recall that any character $\lambda:H\to\RR^\times$ can be used to twist the multiplication of algebras in $\frM$
with the cocycle $\lambda\circ\sigma$: 
for each $\cA$, we denoted by $A^\lambda$ the $\bG$-graded space $\cA$ with the new multiplication $\ast$ given by 
Equation~\eqref{eq:twisted_multiplication}.

Now let the morphisms in $\widetilde{\frM}$ from $\cA$ to $\cA'$ be the pairs $(\psi,\lambda)$ 
where $\lambda$ is a real-valued character on $H$ and $\psi:\cA^\lambda\to\cA'$ is an isomorphism of $\bG$-graded algebras.
Note that the morphisms in $\frM$ can be identified with the pairs that have trivial $\lambda$.

Next we define an extension of $\lpc$ to these new morphisms. Let $\widetilde{\lpc}(\psi,\lambda): \lpc(\cA)\to\lpc(\cA')$
be the isomorphism of $G$-graded spaces given by 
\begin{equation}\label{eq:k*}
a\ot u_g\mapsto\psi(a)\lambda(g\xi(\pi(g))^{-1})\ot u_g\text{ for all }a\in\cA_{\pi(g)},\:g\in G.
\end{equation}
We claim that this is an isomorphism of algebras. Indeed, for any $g_1, g_2\in G$, $a_1\in \cA_{\pi(g_1)}$ and  
$a_2\in \cA_{\pi(g_2)}$, we write $g_1=\xi(\pi(g_1))h_1$,  $g_2=\xi(\pi(g_2))h_2$, for some $h_1, h_2\in H$, 
and let $g=g_1g_2$. Then $u_{g_1}u_{g_2}$ is a scalar multiple of $u_g$ and so $\widetilde{\lpc}(\psi,\lambda)$ sends
$a_1a_2\ot u_{g_1}u_{g_2}$ to $\psi(a_1a_2)\lambda(h)\ot u_{g_1}u_{g_2}$ where
\[
h=g\xi(\pi(g))^{-1}=g_1g_2\xi(\pi(g_1)\pi(g_2))^{-1}=h_1h_2\sigma(\pi(g_1),\pi(g_2)).
\]
Since $\psi(a_1\ast a_2)=\psi(a_1)\psi(a_2)$, we have 
\[
\psi(a_1a_2)=\lambda(\sigma(\pi(g_1),\pi(g_2))^{-1}\psi(a_1)\psi(a_2),
\]
so the image of $a_1a_2\ot u_{g_1}u_{g_2}$ is $\psi(a_1)\psi(a_2)\lambda(h_1)\lambda(h_2)$, 
which is the product of the images of $a_1\ot u_{g_1}$ and $a_2\ot u_{g_1}$.

It is easy to see that $\widetilde{\lpc}(\id_\cA,1)=\id_{\widetilde{\lpc}(\cA)}$ and 
$
\widetilde{\lpc}(\psi'\psi,\lambda'\lambda)=\widetilde{\lpc}(\psi',\lambda')\widetilde{\lpc}(\psi,\lambda),
$
so we have defined a functor $\widetilde{\lpc}:\widetilde{\frM}\to\frN$.

\begin{theorem}\label{th:loop2}
$\widetilde{\lpc}:\widetilde{\frM}\to\frN$ is an equivalence of groupoids.
\end{theorem}

\begin{proof}
Since $\widetilde{\lpc}$ coincides with $\lpc$ on the objects, it is essentially surjective by Theorem \ref{th:loop1}. 
To see that $\widetilde{\lpc}$ is faithful, 
observe that the isomorphism of centroids induced by $\vp=\widetilde{\lpc}(\psi,\lambda): \lpc(\cA)\to\lpc(\cA')$ 
is given by $u_h\mapsto \lambda(h)u_h$ (recall that the centroids are $\lpc(\RR)$, which is the real form of 
$\CC H$ spanned by the elements $u_h$). Indeed,  
\begin{equation}\label{eq:k**}
\begin{split}
\vp(u_h(a\ot u_g))&=\vp(a\ot u_h u_g)=\psi(a)\lambda(hg\xi(\pi(hg))^{-1})\ot u_h u_g\\ 
&=\lambda(h)\psi(a)\lambda(g\xi(\pi(g))^{-1})\ot u_h u_g\\ 
&=\lambda(h)u_h(\vp(a\ot u_g))
\end{split}
\end{equation}
for all $g\in G$, $h\in H$, $a\in\cA_{\pi(g)}$. So, $\lambda$ is uniquely determined by $\vp$, 
hence $\psi$ is also uniquely determined.

It remains to show that $\widetilde{\lpc}$ is full, so suppose that we have an isomorphism $\vp:\lpc(\cA)\to\lpc(\cA')$ 
of $G$-graded algebras. For any $g\in G$, write $\vp(a\ot u_g)=\vp_g(a)\ot u_g$, where $\vp_g:\cA_{\pi(g)}\to\cA_{\pi(g)}'$,
and take $\lambda: H\to\RR^\times$ such that $u_h\mapsto\lambda(h)u_h$ is the isomorphism of centroids induced by $\vp$. 
Define $\psi:\cA\to\cA'$ by $\psi(a)=\vp_{\xi(\bg)}(a)$, for any $\bg\in\bG$ and $a\in\cA_{\bg}$. 
Then $\psi$ is an isomorphism of $\bG$-graded algebras $\cA^\lambda\to\cA'$, and $\widetilde{\lpc}(\psi,\lambda)=\vp$. 
The calculations to establish these claims are similar to the above: 
for the first, put $g_1=\xi(\bg_1)$ and $g_2=\xi(\bg_2)$ in the proof that \eqref{eq:k*} is an isomorphism of algebras, 
and for the second, use \eqref{eq:k**} with $g=\xi(\bg)$.
\end{proof}

\begin{corollary}\label{cor:loop}
Let $\cA$ and $\cA'$ be central simple algebras over $\RR$ equipped with $\bG$-gradings. Then $\lpc(\cA)$ and $\lpc(\cA')$
are isomorphic as $G$-graded algebras if and only if the $\bG$-graded algebra $\cA'$ is isomorphic to a twist of $\cA$.\qed
\end{corollary}

\section{Alternative algebras}\label{s:alternative}

As an application of Theorem~\ref{th:loop1} and Corollary~\ref{cor:loop}, we will obtain the classification of 
real $G$-graded algebras that are graded-central-simple and alternative but not associative: 
as pointed out in Example \ref{ex:2}, these algebras arise from gradings by the quotient groups of $G$ 
on the real octonion algebras $\OO$ and $\OO_s$, and all group gradings on octonion algebras are known. 
For the classification up to isomorphism, the only remaining question is to calculate the twists,
but first we review the relevant facts about octonions.

\subsection{Cayley--Dickson doubling process}\label{ss:CD_process}

Recall that a \emph{Hurwitz algebra} is a unital composition algebra, i.e., a unital algebra equipped with a 
multiplicative nonsingular quadratic form, which is called the \emph{norm} and will be denoted by $n$. Except in the case 
$\chr\FF=2$,``nonsingular'' means that the polar form of $n$, defined by $n(x,y)\bydef n(x+y)-n(x)-n(y)$, is nondegenerate. 
The \emph{standard involution} of a Hurwitz algebra is given by $\bar{x}\bydef n(x,1)1-x$.

It is well known that the dimension of a Hurwitz algebra can be only $1$, $2$, $4$ or $8$. 
Hurwitz algebras of dimension $4$ are referred to as \emph{quaternion algebras} 
and those of dimension $8$ as \emph{octonion} or \emph{Cayley algebras}.

Given an \emph{associative} Hurwitz algebra $\cA$ such that the polar form of $n$ is nondegenerate and any scalar 
$\alpha\in\FF^\times$, let $\CD(\cA,\alpha)$ be the direct sum  of two copies of $\cA$, where we may formally write 
the element $(x,y)$ as $x+yw$, so $\CD(\cA,\alpha)=\cA\oplus \cA w$. 
This is a Hurwitz algebra with multiplication and norm given by:
\begin{align}
&(a+bw)(c+dw)=(ac+\alpha\overline{d}b)+(da+b\overline{c})w,\label{eq:CD_mult}\\
&n\bigl(a+bw\bigr)=n(a)-\alpha n(b).\label{eq:CD_norm}
\end{align} 
For example, the real division algebras of complex numbers, quaternions, and octonions are obtained as 
$\CC=\CD(\RR,-1)$, $\HH=\CD(\CC,-1)$, and $\OO=\CD(\HH,-1)$. Since $\OO$ is not associative, we cannot obtain a Hurwitz 
algebra by doubling it. We will abbreviate $\CD(\cA,\alpha,\beta)\bydef \CD\bigl(\CD(\cA,\alpha),\beta\bigr)$, 
and similarly for $\CD(\FF,\alpha,\beta,\gamma)$. For any $\delta\in\FF^\times$, the mapping 
$(x,y)\mapsto(x,\delta^{-1}y)$ is an isomorphism $\CD(\cA,\alpha)\to\CD(\cA,\alpha\delta^2)$. 
Hence, over $\RR$, the isomorphism class of these algebras depends only on the 
sign of the parameters $\alpha$, $\beta$ and $\gamma$. It turns out that if any of the parameters is positive, we get the 
``split complex numbers'' $\CC_s=\RR\times\RR$, split quaternions $\HH_s\cong M_2(\RR)$, and split octonions $\OO_s$.

Conversely, given any Hurwitz algebra $\cC$ with norm $n$ and a subalgebra $\cA$ such that the restriction to $\cA$ 
of the polar form of $n$ is nondegenerate, and given any nonisotropic element $w\in \cA^\perp$, it follows that 
$n\bigl(\cA,\cA w\bigr)=0$ and that $\cA\oplus \cA w$ is a subalgebra of $\cC$ isomorphic to $\CD(\cA,\alpha)$ 
with $\alpha=-n(w)=w^2$.

\subsection{Gradings on octonion algebras}

Let $\cQ$ be a quaternion subalgebra of a Cayley algebra $\cC$ over a field $\FF$, 
and let $w\in \cQ^\perp$ with $n(w)=-\gamma\neq 0$. Then $\cC =\cQ\oplus \cQ w$ is isomorphic to $\CD(\cQ ,\gamma)$, 
which gives a $\ZZ_2$-grading on $\cC$ with $\cC_{\bar 0}=\cQ$ and $\cC_{\bar 1}=\cQ^\perp=\cQ w$.

In its turn, the quaternion subalgebra $\cQ$ can be obtained from a $2$-dimensional subalgebra $\cK$ 
(either $\FF\times\FF$ or a separable quadratic field extension of $\FF$) as $\cQ =\cK\oplus \cK v\cong\CD(\cK,\beta)$,
with $v\in \cQ \cap \cK^\perp$ and $n(v)=-\beta\neq 0$. Then $\cC $ is isomorphic to $\CD(\cK,\beta,\gamma)$, 
which gives a $\ZZ_2^2$-grading on $\cC$ with $\cC_{(\bar 0,\bar 0)}=\cK$, $\cC_{(\bar 1,\bar 0)}=\cK v$, 
$\cC_{(\bar 0,\bar 1)}=\cK w$, and $\cC _{(\bar 1,\bar 1)}=(\cK v)w$.

If $\chr\FF\neq 2$, then $\cK$ can be obtained by doubling $\FF$: $\cK=\FF\oplus\FF u\cong \CD(\FF,\alpha)$, 
with $u\in \cK\cap\FF^\perp$ and $n(u)=-\alpha\neq 0$, so $\cC\cong\CD(\FF,\alpha,\beta,\gamma)$, which gives a 
$\ZZ_2^3$-grading on $\cC$, with $\deg u=(\bar 1,\bar 0,\bar 0)$, $\deg v=(\bar 0,\bar 1, \bar 0)$, 
and $\deg w=(\bar 0, \bar 0, \bar 1)$.

The above gradings by $\ZZ_2^r$, $r=1,2,3$, are called the \emph{gradings induced by the Cayley--Dickson doubling process}
\cite[p.~131]{EKmon} ($r\ne 3$ if $\chr\FF=2$). 

Finally, over any field, there is a unique (up to isomorphism) split Cayley algebra $\cC _s$, and it admits a 
$\ZZ^2$-grading, called the \emph{Cartan grading}, whose homogeneous components are the eigenspaces for the action 
of a maximal torus of $\AAut_\FF(\cC _s)$.

The following is a stronger version of the main result of \cite{E}, where it was used to classify gradings on Cayley
algebras over algebraically closed fields (see also \cite[Theorem 4.12]{EKmon}).

\begin{theorem}[\cite{EK_G2D4}]\label{th:CD_or_Cartan}
Any nontrivial grading on a Cayley algebra is, up to equivalence, either a grading induced by the 
Cayley--Dickson doubling process starting with a Hurwitz division subalgebra, 
or a coarsening of the Cartan grading on the split Cayley algebra.\qed
\end{theorem}

Starting from this point, a classification of gradings up to isomorphism is obtained in \cite{EK_G2D4} for any Cayley 
algebra. Here we will state the result only for $\FF=\RR$ and $\FF=\CC$. 
To be consistent with our previous notation, we will denote the grading group by $\bG$.

The coarsenings of the Cartan grading on $\OO_s$ are induced by arbitrary homomorphisms $\ZZ^2\to\bG$
(see Subsection~\ref{ss:induced_grading}). 
Let $\bg_1$, $\bg_2$ and $\bg_3$ be the images of the elements $(1,0)$, $(0,1)$, and $(-1,-1)$, respectively, 
so $\bg_1\bg_2\bg_3=\bar e$ (the identity element of $\bG$).
Then the grading is determined by the triple $(\bg_1,\bg_2,\bg_3)$, and two triples yield isomorphic gradings if and only if 
they are in the same orbit under the action of the group $S_3\times\ZZ_2$ (the Weyl group of type $G_2$), where $S_3$ 
permutes the entries of the triple and the generator of $\ZZ_2$ inverts them simultaneously. 
We will denote the resulting graded algebras by $\cC(\bg_1,\bg_2,\bg_3)$. None of them is a graded-division algebra 
(for example, because the identity component contains a nontrivial idempotent).

The remaining nontrivial gradings on $\OO$ and $\OO_s$ are obtained by arbitrary mono\-mor\-phisms $\ZZ_2^r\to\bG$, 
$1\le r\le 3$, from the gradings induced by the Cayley--Dickson doubling process starting from $\HH$ if $r=1$, from 
$\CC$ if $r=2$, and from $\RR$ if $r=3$. All of them are division gradings. To include the trivial grading on $\OO$, 
we will allow $r=0$. The support of the grading is the image of $\ZZ_2^r$, which we denote by $\bT$. 
The parameters used in the doubling process determine a character $\mu:\bT\to\{\pm 1\}$ as follows: 
for any nonzero homogeneous element $x$ of degree $\bar{t}$, we have $x^2\in-\mu(\bar{t})\RR_{>0}$.
The trivial characters give gradings on $\OO$ and the nontrivial ones on $\OO_s$. We will denote the resulting 
graded-division algebras by $\cC(\bT,\mu)$. The algebras corresponding to different pairs $(\bT,\mu)$ are not 
graded-isomorphic.

A similar classification of division gradings is valid for $\HH$ and $\HH_s$ on the one hand and 
$\CC$ and $\CC_s$ on the other hand.
We will denote the resulting graded-division algebras by $\cQ(\bT,\mu)$ and $\cK(\bT,\mu)$, respectively.

Over $\CC$, there is only one Cayley algebra (up to isomorphism), and the classification of gradings is the same as above,
except in the Cayley--Dickson case there is no parameter $\mu$ and $r$ must be equal to $3$. 
We will denote these graded algebras by $\cC_\CC(\bg_1,\bg_2,\bg_3)$ and $\cC_\CC(\bT)$.
Similarly, we also have $\cQ_\CC(\bT)$ with $r=2$ for $M_2(\CC)$ (the complex quaternion algebra) 
and $\cK_\CC(\bT)$ with $r=1$ for $\CC\times\CC$.

\subsection{Graded-simple real alternative algebras}\label{ss:gs_alt}

Let $\cB$ be a $G$-graded algebra over $\RR$ that is graded-simple and alternative but not associative. Assume also 
that the identity component of the centroid $\LL\bydef C(\cB)$ is finite-dimensional, so it is either $\RR$ or $\CC$. 
As before, let $H$ be the support of $\LL$, $\bG=G/H$, and $\pi$ be the natural homomorphism $G\to\bG$. 

If $\LL_e\cong\CC$ then $\LL\cong\CC H$ and, by the ``Correspondence Theorem'' from \cite{ABFP}, we have 
$\cB\cong L_\pi(\cC)$ where $\cC$ is the unique complex Cayley algebra, equipped with a $\bG$-grading, 
and this grading is determined uniquely up to isomorphism. Thus, in this case the algebras $\cB$ can be of two kinds: 
those with $\cC\cong\cC_\CC(\bg_1,\bg_2,\bg_3)$ are classified up to isomorphism by the $(S_3\times\ZZ_2)$-orbits 
of the triples $(\bg_1,\bg_2,\bg_3)$ and those with $\cC\cong\cC_\CC(\bT)$ are classified by the subgroups $\bT\subset\bG$
isomorphic to $\ZZ_2^3$. The graded-division algebras are the latter.

Now let us assume $\LL_e=\RR$ and fix a character $\chi:G\to U\subset\CC^\times$ such that 
$\LL\cong L_\pi^\chi(\RR)$. By Theorem \ref{th:loop1}, we have $\cB\cong L_\pi^\chi(\cC)$ where $\cC$ is either $\OO$
or $\OO_s$, equipped with a $\bG$-grading, so $\cC\cong\cC(\bg_1,\bg_2,\bg_3)$ or $\cC\cong\cC(\bT,\mu)$. 
To apply Corollary~\ref{cor:loop}, we have to classify these $\bG$-graded algebras up to twist.
Let $T=\pi^{-1}(\bT)$, so $\bT=T/H$. Since $\bT$ is an elementary $2$-group (of rank $r$), 
we have a homomorphism $T\to H$ sending $t\mapsto t^2$, which maps $H$ to $H^{[2]}$ and hence induces a homomorphism 
\begin{equation}\label{eq:def_tau} 
\tau:\bT\to H/H^{[2]},\: tH\mapsto t^2 H^{[2]}.
\end{equation}
Note that $\tau$ is the element of $\Hom_\ZZ(\bT,H/H^{[2]})$ corresponding to the abelian group extension $H\to T\to\bT$ 
under the isomorphism $\mathrm{Ext}_\ZZ(\bT,H)\cong\Hom_\ZZ(\bT,H/H^{[2]})$.

\begin{lemma}\label{lm:3}
For any character $\lambda:H\to\RR^\times$, we have $\cC(\bg_1,\bg_2,\bg_3)^\lambda\cong\cC(\bg_1,\bg_2,\bg_3)$ and 
$\cC(\bT,\mu)^\lambda\cong\cC(\bT,\mu')$ where $\mu'(\bt)=\mu(\bt)\lambda_0(\tau(\bt))$ for all $\bt\in\bT$,
and $\lambda_0(hH^{[2]})\bydef\mathrm{sign}\,\lambda(h)$ for all $h\in H$.
\end{lemma}

\begin{proof}
To compute the twists, we fix a section $\xi:\bG\to G$. It is convenient to take $\xi(\bar e)=e$, so that the resulting
$2$-cocycle $\sigma:\bG\times\bG\to H$ is normalized.

The elements $\bg_i$ are the degrees of homogeneous basis elements of one of the Peirce components with respect to 
a nontrivial homogeneous idempotent (while the other Peirce component gives the elements $\bg_i^{-1}$). 
Since the degree of the idempotent is $\bar e$, it remains an idempotent with respect to the twisted multiplication of 
Equation~\eqref{eq:twisted_multiplication}, and its Peirce components remain the same.

Recall that $\mu(\bt)$ is defined by $x^2\in-\mu(\bt)\RR_{>0}$, for any nonzero homogeneous element $x$ of degree 
$\bt$. The twisted multiplication gives:
\[
x\ast x=\lambda(\sigma(\bt,\bt))x^2=\lambda\big(\xi(\bt)^2\xi(\bt^2)^{-1}\big)x^2=\lambda\big(\xi(\bt)^2\big)x^2,
\]   
and the result follows, because 
$\mathrm{sign}\,\lambda\big(\xi(\bt)^2\big)=\lambda_0\big(\xi(\bt)^2 H^{[2]}\big)=\lambda_0(\tau(\bt))$.
\end{proof}

\begin{lemma}\label{lm:4}
Let $\bT_0=\ker\tau$. Then $\cC(\bT,\mu')$ is graded-isomorphic to a twist of $\cC(\bT,\mu)$ if and only if 
$\mu'|_{\bT_0}=\mu|_{\bT_0}$. 
\end{lemma}

\begin{proof}
We have to let $\lambda_0$ in Lemma \ref{lm:3} range over $\Hom_\ZZ(H/H^{[2]},\{\pm 1\})$ 
and consider the resulting $\tau^*(\lambda_0)\bydef\lambda_0\circ\tau$, which are homomorphisms $\bT\to\{\pm 1\}$. 
Consider the exact sequence
\[
1\to\bT_0\to\bT\xrightarrow{\tau}H/H^{[2]}.  
\]
Since these are elementary $2$-groups, we may regard them as vector spaces over the field of $2$ elements and hence 
obtain the following exact sequence: 
\[
\Hom_\ZZ(H/H^{[2]},\{\pm 1\})\xrightarrow{\tau^*}\Hom_\ZZ(\bT,\{\pm 1\})
\xrightarrow{\mathrm{res}}\Hom_\ZZ(\bT_0,\{\pm 1\})\to 1,
\] 
which proves the result.
\end{proof} 

We summarize the results of this subsection:

\begin{theorem}\label{th:alternative_iso}
Let $\cB$ be a $G$-graded algebra over $\RR$ that is graded-simple and alternative but not associative. 
Let $\LL=C(\cB)$ and assume that $\LL_e$ is finite-dimensional, so it is either $\RR$ or $\CC$. 
Let $H$ be the support of $\LL$, $\bG=G/H$, and $\pi$ be the natural homomorphism $G\to\bG$. 
If $\LL_e=\RR$, pick a character $\chi:G\to U\subset\CC^\times$ such that $\LL\cong\lpc(\RR)$.
\begin{enumerate} 
\item[1.] If $\LL_e=\RR$ then $\cB$ is graded-isomorphic to one of the following:
\begin{enumerate}
\item[(a)] $\lpc(\cC(\bg_1,\bg_2,\bg_3))$ for a triple $(\bg_1,\bg_2,\bg_3)\in\bG^3$ with $\bg_1\bg_2\bg_3=\bar e$;
\item[(b)] $\lpc(\cC(\bT,\mu))$ for a subgroup $\bT\subset\bG$ isomorphic to $\ZZ_2^r$, $0\le r\le 3$, and a character $\mu:\bT\to\{\pm 1\}$.
\end{enumerate}
\item[2.] If $\LL_e\cong\CC$ then $\cB$ is graded-isomorphic to one of the following:
\begin{enumerate} 
\item[(a)] $L_\pi(\cC_\CC(\bg_1,\bg_2,\bg_3))$ for a triple $(\bg_1,\bg_2,\bg_3)\in\bG^3$ with $\bg_1\bg_2\bg_3=\bar e$;
\item[(b)] $L_\pi(\cC_\CC(\bT))$ for a subgroup $\bT\subset\bG$ isomorphic to $\ZZ_2^3$.
\end{enumerate}
\end{enumerate}
Two algebras in different items are not graded-isomorphic to one another. Two algebras in item 1(a) or 2(a) are graded-isomorphic if and only if 
their triples belong to the same $(S_3\times\ZZ_2)$-orbit. Two algebras in item 2(b) are graded-isomorphic if and only if their subgroups 
are equal. Finally, two algebras in item 1(b) are graded-isomorphic if and only if their subgroups are equal and their characters have the same restriction
to $\bT_0\bydef\ker\tau$ where $\tau:\bT\to H/H^{[2]}$ is given by Equation~\eqref{eq:def_tau}.\qed 
\end{theorem}

\subsection{Graded-division real alternative algebras}\label{ss:ds_alt}

The algebras in items 1(b) and 2(b) of Theorem~\ref{th:alternative_iso} are characterized by the property that they are
graded-division algebras (see Remark \ref{rem:div}). They can be obtained by 
the Cayley--Dickson doubling process over the graded-field $\LL$ as follows. First of all, the formulas 
\eqref{eq:CD_mult} and \eqref{eq:CD_norm} make sense for unital algebras with a quadratic form over any 
commutative associative ring, in particular over any graded-field $\LL$, but we also have to keep track of the gradings.
To this end, suppose $\cA$ is a graded $\LL$-module with an $\LL$-bilinear multiplication and a quadratic form
$n:\cA\to\LL$ such that $\cA$ is a unital graded algebra and $n(\cA_{g_1},\cA_{g_2})\subset\LL_{g_1g_2}$ 
for all $g_1,g_2\in G$, so the mapping $x\mapsto\bar{x}$ is degree-preserving. Given $k\in G$ and $\alpha\in\LL_{k^2}$, 
Equation~\eqref{eq:CD_mult} makes $\cB\bydef\CD(\cA,\alpha)$ a graded algebra if we define
\begin{equation}\label{eq:CD_deg}
\CD(\cA,\alpha)_g=\cA_g\oplus\cA_{gk^{-1}}w\text{ for all }g\in G.
\end{equation}
Note that $w=(0,1)$ is assigned degree $k$. Moreover, the polar form of the quadratic form $n$ defined by 
Equation~\eqref{eq:CD_norm}, satisfies the condition $n(\cB_{g_1},\cB_{g_2})\subset\LL_{g_1g_2}$ for all $g_1,g_2\in G$.
We will denote the resulting unital graded algebra with a quadratic form by $\CD(\cA,(\alpha,k))$.

For example, the $\bG$-graded algebras $\cC(\bT,\mu)$ can be obtained as follows. Take $\LL=\RR$ with trivial grading.
In the case $r=3$, if $\{\bt_1,\bt_2,\bt_3\}$ is a basis of the elementary $2$-group $\bT$, then 
\begin{equation}\label{eq:CTmu_CD}
\cC(\bT,\mu)\cong\CD\big(\RR,(-\mu(\bt_1),\bt_1),(-\mu(\bt_2),\bt_2),(-\mu(\bt_3),\bt_3)\big),
\end{equation}
where we use an abbreviation for the iterated Cayley--Dickson doubling process similar to what we had in 
Subsection~\ref{ss:CD_process}. The isomorphism \eqref{eq:CTmu_CD} remains valid in the case $r<3$ if we make the convention 
that $\bt_j=\bar{e}$ for $1\le j\le 3-r$, while the $\bt_j$ with $3-r<j\le 3$ form a basis of $\bT$.

Now let us see how the loop algebra construction interacts with the (graded) Cayley--Dickson doubling process.

\begin{lemma}\label{lm:5}
Let $\cA$ be a unital $\bG$-graded algebra with a quadratic form over $\RR$, $\LL=\lpc(\RR)$, $\alpha\in\RR^\times$ and 
$\bar{k}\in\bG$ with $\bar{k}^2=\bar{e}$. Then 
\[
\lpc\big(\CD(\cA,(\alpha,\bar{k}))\big)\cong\CD\big(\lpc(\cA),(\alpha u_k^2,k)\big),
\]
for any $k\in G$ satisfying $\pi(k)=\bar{k}$.
\end{lemma}

\begin{proof}
In view of \eqref{eq:CD_deg}, applied first with $\bG$ and then with $G$, 
we have an isomorphism of $G$-graded vector spaces 
$\lpc\big(\CD(\cA,(\alpha,\bar{k}))\big)\to\CD\big(\lpc(\cA),(\alpha u_k^2,k)\big)$ given by 
\[
(a+bw)\otimes u_g\mapsto a\otimes u_g+(b\otimes u_g u_k^{-1})w\text{ for all }
a\in\cA_{\pi(g)},\, b\in\cA_{\pi(g)\bar{k}^{-1}},\, g\in G.
\]
This is also an isomorphism of $\LL$-algebras with quadratic forms because it is the restriction of the mapping 
$\CD(\cA,\alpha)\otimes\cD\to\CD(\cA\otimes\cD,\alpha u_k^2)$, where $\cD=\RR^\gamma G$, obtained by composing 
the $\cD$-algebra isomorphism $\CD(\cA,\alpha)\otimes\cD\to\CD(\cA\otimes\cD,\alpha)$ sending 
$(a+bw)\otimes u_g\mapsto a\otimes u_g+(b\otimes u_g)w$ and the $\cD$-algebra isomorphism 
$\CD(\cA\otimes\cD,\alpha)\to\CD(\cA\otimes\cD,\alpha u_k^2)$ sending $(x,y)\mapsto (x,yu_k^{-1})$,
for all $x,y\in\cA\otimes\cD$.
\end{proof}

Applying this lemma repeatedly, we obtain from \eqref{eq:CTmu_CD} the following models for the algebras in 
item 1(b) of Theorem \ref{th:alternative_iso}:
\[
\lpc(\cC(\bT,\mu))\cong\CD\big(\LL,(-\mu(\bt_1)u_{t_1}^2,t_1),(-\mu(\bt_2)u_{t_2}^2,t_2),(-\mu(\bt_3)u_{t_3}^2,t_3)\big),
\]
for any $t_j\in G$ satisfying $\pi(t_j)=\bt_j$. Note that $u_{t_j}^2=\gamma(\bt_j,\bt_j)u_{t_j^2}\in\LL$.

This can be simplified with a special choice of the elements $t_j$.
Indeed, let $r_0$ be the rank of $\bT_0$ and pick the basis of $\bT$ so that the $\bt_j$ with $3-r<j\le 3-r+r_0$ 
form a basis of $\bT_0$. Then the $t_j$ with $3-r< j\le 3-r+r_0$ can be chosen to satisfy $t_j^2=e$, so that 
the subgroup $T_0$ of $G$ generated by these elements is mapped by $\pi$ isomorphically onto $\bT_0$. Moreover, for 
these $t_j$, we have $u_{t_j}^2=\chi(t_j)\in\{\pm 1\}$. Unless $\LL\cong\RR H$, there exists $h\in H_{[2]}$ such that 
$\chi(h)=-1$ (see Lemma~\ref{lm:1}), which can be used to adjust the $t_j$ to make $u_{t_j}^2=\mu(\bt_j)$.
If $\LL\cong\RR H$ then we can take $\chi$ to be trivial, so $u_{t_j}^2=1$. Now, for $1\le j\le 3-r$, we will take $t_j=e$. 
As to the remaining case, $3-r+r_0<j\le 3$, it does not matter how we choose these $t_j$ because, in view of the isomorphism 
condition in item 1(b), we may change $\mu(\bt_j)$ arbitrarily without affecting the isomorphism class of 
the graded algebra. Thus we obtain (cf. Remark \ref{rem:mod_over_centroid_nonsplit}):
\begin{equation}\label{eq:LCTmu_CD}
\lpc(\cC(\bT,\mu))\cong\begin{cases}
\cC(T_0,\mu_0)\otimes\LL & \text{if } r_0=r;\\
\CD\big(\cQ(T_0,\mu_0)\otimes\LL,(\pm u_{t_3^2},t_3)\big) & \text{if } r_0=r-1;\\
\CD\big(\cK(T_0,\mu_0)\otimes\LL,(\pm u_{t_2^2},t_2),(\pm u_{t_3^2},t_3)\big) & \text{if } r_0=r-2;\\
\CD\big(\LL,(\pm u_{t_1^2},t_1),(\pm u_{t_2^2},t_2),(\pm u_{t_3^2},t_3)\big) & \text{if } r_0=r-3;\\
\end{cases}
\end{equation}
where the signs can be chosen arbitrarily, and $\mu_0=\mu\circ\pi|_{T_0}$ if $\LL\cong\RR H$ and trivial otherwise. 

Similar models can be obtained for $L_\pi(\cC_\CC(\bT))$ in item 2(b) of Theorem \ref{th:alternative_iso}, 
but we have $\LL\cong\CC H$, there is no $\mu_0$, $r$ must be $3$, and we have $t_j^2$ instead of $\pm u_{t_j^2}$.

We will now obtain a classification of graded-division real alternative algebras up to equivalence. First of all,
we replace the grading group $G$ by the support $T$, so equivalence becomes the same as weak isomorphism 
(see Subsection \ref{ss:wiso_equiv}), and we have to investigate the effect of the automorphisms of the group $T$ on 
the graded algebras $\cB$ in items 1(b) and 2(b). For any $\alpha\in\Aut(T)$, we have a canonical isomorphism of graded
algebras ${}^\alpha C(\cB)\cong C({}^\alpha\cB)$. Hence, we fix a representative in each $\Aut(T)$-orbit of 
the graded-isomorphism classes of graded-fields (over the group $T$) so that we may consider the graded-isomorphism class of 
$\LL=C(\cB)$ fixed and restrict ourselves to those $\alpha\in\Aut(T)$ that satisfy ${}^\alpha\LL\cong\LL$. 
In item 2(b), this latter condition amounts to $\alpha(H)=H$, while in item 1(b) it says more: 
$\alpha(H)=H$ and $\chi\circ\alpha|_{H_{[2]}}=\chi|_{H_{[2]}}$ (see Lemma \ref{lm:1}). 
In any case, $\alpha$ induces an automorphism $\bar{\alpha}$ of $\bT=T/H$, and we have $\bar{\alpha}(\bT_0)=\bT_0$.

Consider $\cB$ in item 1(b), so $\cB\cong\lpc(\cC(\bT,\mu))$. Then ${}^\alpha\cB$ is given by the right-hand side of 
\eqref{eq:LCTmu_CD} in which $\LL$ is replaced by $\LL'\bydef{}^\alpha\LL$, 
$T_0$ by $T_0'\bydef\alpha(T_0)$, $\mu_0$ by $\mu_0'\bydef\mu_0\circ\alpha^{-1}|_{T_0'}$, and $(\pm u_{t_j^2},t_j)$ by 
$(\pm u_{t_j^2},\alpha(t_j))$. Note that $\LL'_{\alpha(h)}=\LL_h=\RR u_h$ for all $h\in H$, and there exists an
isomorphism of graded algebras $\psi:\LL'\to\LL$, which must therefore send $u_h$ to either $u_{\alpha(h)}$ or 
to $-u_{\alpha(h)}$ (depending on $h$). Extending the isomorphism $\id\otimes\psi$ to the Cayley--Dickson doubles, 
we obtain
\[
{}^\alpha\cB\cong\begin{cases}
\cC(T_0',\mu_0')\otimes\LL & \text{if } r_0=r;\\
\CD\big(\cQ(T_0',\mu_0')\otimes\LL,(\pm u_{(t'_3)^2},t'_3)\big) & \text{if } r_0=r-1;\\
\CD\big(\cK(T_0',\mu_0')\otimes\LL,(\pm u_{(t'_2)^2},t'_2),(\pm u_{(t'_3)^2},t'_3)\big) & \text{if } r_0=r-2;\\
\CD\big(\LL,(\pm u_{(t'_1)^2},t'_1),(\pm u_{(t'_2)^2},t'_2),(\pm u_{(t'_3)^2},t'_3)\big) & \text{if } r_0=r-3;\\
\end{cases}
\]
where $t'_j=\alpha(t_j)$. These elements have the same properties as the $t_j$ in \eqref{eq:LCTmu_CD}, so the net 
effect of $\alpha$ on the graded-isomorphism class of $\cB$ is the replacement of $(T_0,\mu_0)$ by $(T_0',\mu_0')$.
Now, $T_0$ can be mapped onto any other complement of $H$ in $\pi^{-1}(\bT_0)$ by an automorphism of the latter that 
restricts to the identity on $H$ and therefore extends to an automorphism $\alpha$ of $T$ as above. 
Moreover, any automorphism of $T_0$ extends in the same manner to an automorphism $\alpha$. It follows that there is 
only one equivalence class of $\cB$ when we have fixed $T$ and the graded-isomorphism class of $\LL$ other than $\RR H$, 
and two equivalence classes with a fixed $T$ and $\LL\cong\RR H$: one corresponding to $\mu_0=1$ and 
the other to $\mu_0\ne 1$, where $1$ denotes the trivial character.

For $\cB$ in item 2(b), i.e., $\cB\cong L_\pi(\cC_\CC(\bT))$, we can use the same arguments or apply the canonical
isomorphism ${}^\alpha(L_\pi(\cA))\cong L_\pi({}^{\bar{\alpha}}\cA)$ defined by $a\otimes g\mapsto a\otimes\alpha(g)$,
for any $\bG$-graded algebra $\cA$. (This isomorphism could also be applied for item 1(b) in the case of split centroid.) 

To summarize our classification, it is convenient to introduce the following notation. Let $T$ be an abelian group with 
a subgroup $H$ such that $\bT\bydef T/H$ is an elementary $2$-group of rank $r\in\{0,1,2,3\}$. Let $\cO$ be an orbit 
of nontrivial characters $H_{[2]}\to\{\pm 1\}$ under the action of the stabilizer of $H$ in $\Aut(T)$. For each such triple
$(T,H,\cO)$, we fix a representative $\chi_0\in\cO$ and extend it to a character $\chi:T\to U\subset\CC^\times$. Let
\[
\mathcal{DA}(T,H,\cO)\bydef\lpc(\cC(\bT,1)),
\]
where $\pi:T\to\bT$ is the natural homomorphism. For each pair $(T,H)$ as above, fix a character $\mu:\bT\to\{\pm 1\}$ 
whose restriction to $\bT_0\bydef\ker\tau$ is nontrivial, where $\tau:\bT\to H/H^{[2]}$ is given by 
Equation~\eqref{eq:def_tau}, and let
\[
\mathcal{DA}(T,H)\bydef L_\pi(\cC(\bT,1))\text{ and }\mathcal{DA}(T,H)'\bydef L_\pi(\cC(\bT,\mu)).
\]
(If $\bT_0$ is trivial then only $\mathcal{DA}(T,H)$ is defined.) Finally, for each pair $(T,H)$ with $r=3$, let
\[
\mathcal{DA}_\CC(T,H)\bydef L_\pi(\cC_\CC(\bT)).
\]
Then we have the following:

\begin{theorem}\label{th:alternative_equ}
Let $\cB$ be a real graded-division algebra with support $T$ that is alternative but not associative. 
Let $\LL=C(\cB)$ and assume that $\LL_e$ is finite-dimensional, so it is either $\RR$ or $\CC$, 
and let $H$ be the support of $\LL$.
\begin{enumerate} 
\item[1.] If $\LL_e=\RR$ and $\LL\not\cong\RR H$, 
then $\cB$ is equivalent to $\mathcal{DA}(T,H,\cO)$ for a unique orbit $\cO$
of nontrivial characters $H_{[2]}\to\{\pm 1\}$ under the action of the stabilizer of $H$ in $\Aut(T)$.
\item[2.] If $\LL_e=\RR$ and $\LL\cong\RR H$, then $\cB$ is equivalent to $\mathcal{DA}(T,H)$ or $\mathcal{DA}(T,H)'$, 
but not both.
\item[3.] If $\LL_e\cong\CC$, then $\cB$ is equivalent to $\mathcal{DA}_\CC(T,H)$.\qed
\end{enumerate}
\end{theorem}

We can make the classification more explicit in the finite-dimensional case by classifying the pairs $(T,H)$ and the triples 
$(T,H,\chi_0)$ up to isomorphism. First of all, we have a canonical decomposition $T=T'\times T''$ where $T'$ is a $2$-group
and $T''$ has odd order. Then $H=H'\times T''$ and $H_{[2]}=H'_{[2]}$, so the problem reduces to $2$-groups. 

A subset $\{g_1,\ldots,g_s\}$ of a finite abelian group $G$ will be called a \emph{basis} if the orders of $g_j$ 
are nontrivial prime powers and $G=\langle g_1\rangle\times\cdots\times\langle g_s\rangle$. 
If $G$ is an elementary $p$-group, this is indeed a basis of $G$ as a vector space over the field of $p$ elements. 
We leave the proof of the following result to the reader.

\begin{lemma}\label{lm:6}
Let $\pi:G\to\bG$ be an epimorphism of finite abelian $p$-groups where $\bG$ is elementary. 
Then there exists a basis $\{g_1,\ldots,g_s\}$ of $G$ such that the elements $\pi(g_j)$ that are different from 
the identity form a basis of $\bG$.\qed
\end{lemma}

In particular, our pairs $(T,H)$ are described as follows: there is a basis of $T$ in which $r$ elements are marked, 
the orders of the marked elements are powers of $2$, and $H$ is generated by the unmarked elements and 
the squares of the marked elements. Up to isomorphism, a pair $(T,H)$ is represented by a tuple of nontrivial prime powers 
(the orders of the basis elements) in which $r$ powers of $2$ are marked: for example, the pair 
$(\ZZ_2\times\ZZ_2\times\ZZ_4,0\times\ZZ_2\times 2\ZZ_4)$ is represented by $(\underline{2},2,\underline{4})$. 
Two pairs are isomorphic if and only if the tuples are the same up to permutation. 
(The ``only if'' part can be seen by looking at the quotients $T^{[2^k]}/H^{[2^k]}$ for $k=1,2,\ldots$)

Now we want to include $\chi_0$ into consideration. The classification of pairs $(H,\chi_0)$ up to isomorphism is equivalent 
to the classifications of graded-fields in Example \ref{ex:1} up to equivalence (which was done in \cite{BZGRDA}).
Such a pair is described by the tuple consisting of the orders of the basis elements where each power of $2$ 
is given a sign according to the value of $\chi_0$ on the unique element of order $2$ in the cyclic group generated 
by the basis element. Since we consider nontrivial $\chi_0$, at least one entry of the tuple must have negative sign. 
It turns out that, by changing the basis, one can always achieve that there is exactly one negative sign. This is done 
using the following moves: if two distinct basis elements $h_i$ and $h_j$ with $o(h_i)=2^m\le 2^n=o(h_j)$ give 
negative sign, we replace $h_i$ with $h_i h_j^{2^{n-m}}$, which gives positive sign. Hence, a pair $(H,\chi_0)$
is represented, up to isomorphism, by a tuple of nontrivial prime powers in which exactly one power of $2$ is given 
negative sign: for example, the tuple $(2,-8)$ represents the pair $(\ZZ_2\times\ZZ_8,\chi_0)$ where 
$\chi_0((\bar{1},\bar{0}))=1$ and $\chi_0((\bar{0},\bar{4}))=-1$. 
Two pairs are isomorphic if and only if the tuples are the same up to permutation. 
(The ``only if'' part can be seen by looking at the restrictions of $\chi_0$ to the subgroups $H^{[2^k]}\cap H_{[2]}$ 
for $k=1,2,\ldots$) Hence, every graded-field in Example \ref{ex:1} is equivalent to either $\RR H$ (corresponding 
to the trivial $\chi_0$) or $\RR K\ot(-n)$ where $H\cong K\times\ZZ_{2^n}$ and the $\ZZ_{2^{|m|}}$-graded 
algebra $(m)$ is defined for any $m\in\ZZ$ as follows: 
\[
(n)\bydef\RR[x]/(x^{2^n}-1)\text{ and }(-n)\bydef\RR[x]/(x^{2^n}+1)\text{ where }\deg x=\bar{1}\in\ZZ_{2^n}.
\]
Note that $(0)=\RR$ (trivially graded) and, if $H$ is a $2$-group, $\RR H$ is equivalent to the tensor product of 
graded algebras of the form $(n)$, $n\in\NN$.

Finally, for the triples $(T,H,\chi_0)$, we combine the previous arguments: we have a basis of $H$ in which $r_0\le r$ 
elements are marked (being the nontrivial squares of the marked basis elements of $T$), and each 
basis element gives a sign. We can reduce the number of negative signs to one by applying the same moves as above to the 
basis elements of $T$, except that we are not allowed to replace $t_i$ with $t_i t_j^{2^{n-m}}$ if $t_i$ is in $H$, 
$t_j$ is not in $H$, and $m=n$. This is not an obstacle, because in this situation we can replace $t_j$ with $t_i t_j$.
Therefore, the triples $(T,H,\chi_0)$ are classified by tuples of nontrivial prime powers in which $r$ powers of $2$ 
are marked and one of the powers of $2$ (marked or unmarked) is given negative sign. 
This labeling is arbitrary except that a marked power of $2$ can be given negative sign only if it is at least $4$.

Now we can describe more explicitly the finite-dimensional graded algebras in Theorem \ref{th:alternative_equ}. For 
integers $n_1,n_2,n_3$, consider $T=\ZZ_{2^{|n_1|+1}}\times\ZZ_{2^{|n_2|+1}}\times\ZZ_{2^{|n_3|+1}}$ and 
its standard basis $t_1=(\bar{1},\bar{0},\bar{0})$, $t_2=(\bar{0},\bar{1},\bar{0})$, $t_3=(\bar{0},\bar{0},\bar{1})$.
Define the following $T$-graded algebra:
\[
\RR\{n_1,n_2,n_3\}\bydef\CD\big(\LL,(x_1,t_1),(x_2,t_2),(x_3,t_3)\big)\text{ where }
\LL=(n_1)\otimes(n_2)\otimes(n_3),
\]
with $x_j$ being the generator of the $j$-th factor of $\LL$, with $\deg x_j=2t_j$. Thus, the Cayley--Dickson generators 
satisfy $u^2=x_1$, $v^2=x_2, w^2=x_3$ and are assigned degrees $\deg u=t_1$, $\deg v=t_2$, $\deg w=t_3$. Here we make 
the convention that $x_j=-1$ if $n_j=0$. The algebra $\RR\{n_1,n_2,n_3\}$ falls into 
item 1 of Theorem \ref{th:alternative_equ} if at least one $n_j<0$ (we may assume without loss of generality 
that exactly one) and into item 2 otherwise. The pair $(T,H)$ and the orbit of $\chi_0$  
are represented by the tuple $(\pm\underline{2}^{|n_1|+1},\pm\underline{2}^{|n_2|+1},\pm\underline{2}^{|n_3|+1})$ 
where the signs are given as follows: plus if $n_j\ge 0$ and minus if $n_j<0$. 
We define in a similar manner the graded algebras
\begin{align*}
\CC\{n_2,n_3\}&\bydef\CD\big(\CC\otimes\LL,(x_2,t_2),(x_3,t_3)\big)\text{ where }\LL=(n_2)\otimes(n_3);\\
\HH\{n_3\}&\bydef\CD\big(\HH\otimes\LL,(x_3,t_3)\big)\text{ where }\LL=(n_3).
\end{align*}
For nonnegative integers $n_1,n_2,n_3$, we also define the graded algebras $\RR\{n_1,n_2,n_3\}'$, $\CC\{n_2,n_3\}'$ 
and $\HH\{n_3\}'$ as above, but with the convention $x_j=1$ if $n_j=0$. 
They fall into item 2 of Theorem \ref{th:alternative_equ}. Note that in our notation the letters $\RR$, $\CC$ and $\HH$
indicate the identity component. We can also have $\OO$ as the identity component. In summary:

\begin{corollary}\label{cor:alternative_equ}
Let $\cB$ be a finite-dimensional real graded-division algebra that is alternative but not associative. 
\begin{enumerate} 
\item[1.] If $\cB$ is graded-central with a nonsplit centroid, 
then $\cB$ is equivalent to one of the following:
\begin{enumerate}
\item[(a)] $\RR K\otimes\RR\{n_1,n_2,n_3\}$, $\RR K\otimes\CC\{n_2,n_3\}$, or $\RR K\otimes\HH\{n_3\}$,
where $K$ is a finite abelian group and exactly one of the integers $n_j$ is negative;
\item[(b)] $\RR K\otimes(-n)\otimes\RR\{n_1,n_2,n_3\}$, $\RR K\otimes(-n)\otimes\CC\{n_2,n_3\}$, 
$\RR K\otimes(-n)\otimes\HH\{n_3\}$, or $\RR K\otimes(-n)\otimes\OO$, where $K$ is a finite abelian group, 
$n>0$, and $n_j\ge 0$.
\end{enumerate}
\item[2.] If $\cB$ is graded-central with a split centroid, 
then $\cB$ is equivalent to one of the following:
\begin{enumerate}
\item[(a)] $\RR K\otimes\RR\{n_1,n_2,n_3\}$, $\RR K\otimes\CC\{n_2,n_3\}$, $\RR K\otimes\HH\{n_3\}$, or $\RR K\otimes\OO$,
where $K$ is a finite abelian group and $n_j\ge 0$;
\item[(b)] $\RR K\otimes\RR\{n_1,n_2,n_3\}'$, $\RR K\otimes\CC\{n_2,n_3\}'$, or $\RR K\otimes\HH\{n_3\}'$,
where $K$ is a finite abelian group and $n_j\ge 0$ with at least one being $0$.
\end{enumerate}
\item[3.] If $\cB$ is not graded-central, then $\cB$ is equivalent to $\CC K\otimes\RR\{n_1,n_2,n_3\}$,
where $K$ is a finite abelian group and $n_j\ge 0$.
\end{enumerate}
Two algebras in different items are not equivalent to one another. Two algebras in the same item are equivalent 
if and only if they have isomorphic abelian groups $K$ and, if applicable, the same value of $n$ and the same $n_j$
up to permutation.\qed
\end{corollary}

\section{Jordan algebras of degree $2$}\label{s:Jordan} 

As another application of Theorem~\ref{th:loop1} and Corollary~\ref{cor:loop}, we will classify the finite-dimensional 
real $G$-graded unital Jordan algebras that are graded-simple and have degree $2$ over the center (which is identified with the centroid).

The concept of \emph{degree} is introduced for any finite-dimensional strictly power-associative unital algebra over a field  
using the so-called \emph{generic minimal polynomial} (see e.g. \cite[p.~223]{J}), and this can be extended to algebras over 
direct products of fields: indeed, such an algebra is the direct product of algebras over fields, and we take the maximum degree of the factors. 

\subsection{Jordan algebras of symmetric bilinear forms}

Let $V$ be a vector space over a field $\FF$ ($\chr{\FF}\ne 2$) and let $b:V\times V\to\FF$ be a symmetric bilinear form. 
Then $\cJ(V,b)\bydef\FF 1\oplus V$, with multiplication defined by $uv=b(u,v) 1$ for all $u,v\in V$, is a unital Jordan algebra. 
It is known \cite[p.~207]{J} that any central simple Jordan algebra of degree $2$ over $\FF$ 
is isomorphic to $\cJ(V,b)$ where $\dim V\ge 2$ and $b$ is nondegenerate. Moreover, the pair $(V,b)$ is determined up to isometry. 

As observed in \cite{BS}, all $G$-gradings on these Jordan algebras are obtained from 
$G$-gradings on $V$ satisfying $b(V_{g_1},V_{g_2})=0$ for all $g_1,g_2\in G$ with $g_1g_2\ne e$. 

More generally, if $\LL$ is a graded-field, $\cV$ is a graded $\LL$-module (hence free), and $b:\cV\times\cV\to\LL$ is a symmetric $\LL$-bilinear form satisfying 
$b(\cV_{g_1},\cV_{g_2})\subset\LL_{g_1g_2}$ for all $g_1,g_2\in G$, then $\cJ(\cV,b)\bydef\LL 1\oplus\cV$, with multiplication defined as above and 
$\cJ(\cV,b)_g\bydef\LL_g 1\oplus\cV_g$, is a graded Jordan algebra over $\LL$. 

\subsection{Graded-simple real Jordan algebras of degree 2}
 
Let $\cB$ be a finite-dimensional graded-simple real Jordan algebra, let $\LL$ be the centroid of $\cB$, and let $H$ be the support of $\LL$.
Then $H$ is finite, $\LL_e$ is either $\RR$ or $\CC$, and $\cB\cong L_\pi(\cA)$ or $\cB\cong\lpc(\cA)$ where $\cA$ is a finite-dimensional central simple 
Jordan algebra over $\RR$ or $\CC$, respectively, $\pi$ is the natural homomorphism $G\to\bG\bydef G/H$, 
and $\chi$ is a character $G\to U\subset\CC^\times$, as before. 
Since $\cA$ is unital (see e.g. \cite[p.~201]{J}), so is $\cB$, and we can identify $\LL$ with the center of $\cB$. Disregarding the grading, 
$\LL$ is isomorphic to the direct product of copies of $\RR$ and $\CC$ (see Example \ref{ex:1}). 

Now assume that $\cB$ has degree $2$ over $\LL$. If $\cB\cong L_\pi(\cA)$ then $\LL$ is the group algebra of $H$ and $\cA$ is the quotient of $\cB$ induced by 
the augmentation map of $\LL$, hence $\cA$ also has degree $2$. (It cannot be $1$, because in this case $\cA$ would be the ground field and hence $\cB=\LL$.) 
The same conclusion is obtained if $\cB\cong\lpc(\cA)$ by extending scalars from $\RR$ to $\CC$. Therefore, $\cA\cong\cJ(V,b)$ as above, with a $\bG$-grading 
coming from a $\bG$-grading on $V$ satisfying $b(V_{\bg_1},V_{\bg_2})=0$ for all $\bg_1,\bg_2\in G$ with $\bg_1\bg_2\ne \bar{e}$. 

The pair $(V,b)$ is determined up to graded isometry (see \cite[Theorem 5.42]{EKmon}). Over $\CC$, such pairs are parametrized by certain functions 
$\kappa:\bG\to\ZZ_{\ge 0}$. Indeed, the condition on $b$ implies that we have an orthogonal decomposition as follows:
\[
V=V_{\bg_1}\oplus\cdots\oplus V_{\bg_m}\oplus(V_{\bg'_{m+1}}\oplus V_{\bg''_{m+1}})\oplus\cdots\oplus (V_{\bg'_{l}}\oplus V_{\bg''_{l}}), 
\]
where all elements of $\bG$ are distinct and, for $1\le j\le m$, we have $\bg_j^2=\bar{e}$ and the restriction of $b$ to the subspace $V_{\bg_j}$ is nondegenerate, 
while for $m<j\le l$, we have $\bg'_j\bg''_j=\bar{e}$ and the subspaces $V_{\bg'_{j}}$ and $V_{\bg''_{j}}$ are totally isotropic and in duality through $b$. 
Therefore, we can choose bases in these subspaces so that $b$ is represented by the following matrix:
\[
I_{k_1}\oplus\cdots\oplus I_{k_m}\oplus
\begin{bmatrix}
0&I_{k_{m+1}}\\
I_{k_{m+1}}&0
\end{bmatrix}
\oplus\cdots\oplus
\begin{bmatrix}
0&I_{k_l}\\
I_{k_l}&0
\end{bmatrix}
\]
where $k_j=\dim V_{\bg_j}$ for $1\le j\le m$ and $k_j=\dim V_{\bg'_j}=\dim V_{\bg''_j}$ for $m<j\le l$ (see \cite[p.~199]{EKmon}).
We define 
\[
\kappa(\bg)\bydef\dim V_\bg\:\text{ for all }\bg\in\bG,
\]
so the support of $\kappa$ is $\{\bg_1,\ldots,\bg_m,\bg'_{m+1},\bg''_{m+1},\ldots,\bg'_l,\bg''_l\}$ and $\kappa$ is \emph{balanced} in the sense that 
$\kappa(\bg^{-1})=\kappa(\bg)$ for all $\bg\in\bG$. Conversely, every balanced $\kappa$ with finite support corresponds to $(V,b)$ as above.
We will denote the corresponding $\bG$-graded Jordan algebra by $\cJ_\CC(\kappa)$.

Over $\RR$, we have a similar situation, but $b$ will be represented by a matrix of the form
\[
I_{p_1,q_1}\oplus\cdots\oplus I_{p_m,q_m}\oplus
\begin{bmatrix}
0&I_{k_{m+1}}\\
I_{k_{m+1}}&0
\end{bmatrix}
\oplus\cdots\oplus
\begin{bmatrix}
0&I_{k_l}\\
I_{k_l}&0
\end{bmatrix}
\]
where $I_{p,q}=\matr{I_p&0\\0&-I_q}$, so $p_j-q_j$ is the signature of the restriction of $b$ to $V_{\bg_j}$, $1\le j\le m$.
Thus, we have one more invariant, namely, the function $\sigma:\bG_{[2]}\to\ZZ$ defined by 
\[
\sigma(\bg)\bydef\mathrm{signature}(b|_{V_\bg})\:\text{ for all }\bg\in\bG_{[2]},
\]
which satisfies $|\sigma(\bg)|\le\kappa(\bg)$ and $\sigma(\bg)\equiv\kappa(\bg)\pmod{2}$ for all $\bg\in\bG_{[2]}$.
For a given $\kappa$, we will refer to functions $\sigma$ satisfying these conditions as \emph{signature functions}.
Every such pair $(\kappa,\sigma)$ determines $(V,b)$ as above.
We will denote the corresponding $\bG$-graded Jordan algebra by $\cJ(\kappa,\sigma)$.

Similarly to Subsection \ref{ss:gs_alt}, we define a homomorphism 
\begin{equation}\label{eq:def_tau_again} 
\tau:\bG_{[2]}\to H/H^{[2]},\: gH\mapsto g^2 H^{[2]}.
\end{equation}
Then the cocycle twists of $\cJ(\kappa,\sigma)$ are calculated similarly to Lemma~\ref{lm:3}:

\begin{lemma}\label{lm:7}
For any character $\lambda:H\to\RR^\times$, we have $\cJ(\kappa,\sigma)^\lambda\cong\cJ(\kappa,\sigma')$ where $\sigma'(\bg)=\sigma(\bg)\lambda_0(\tau(\bg))$ 
for all $\bg\in\bG_{[2]}$, and $\lambda_0(hH^{[2]})\bydef\mathrm{sign}\,\lambda(h)$ for all $h\in H$.\qed 
\end{lemma}

Therefore, $\cJ(\kappa,\sigma')$ is graded-isomorphic to a twist of $\cJ(\kappa,\sigma)$ if and only if $\sigma'=\sigma\tilde{\mu}$
for some character $\tilde{\mu}:\bG_{[2]}\to\{\pm 1\}$ that restricts to the trivial character on $\ker\tau$ (see the proof of Lemma~\ref{lm:4}).
To summarize:

\begin{theorem}\label{th:J_deg2_iso}
Let $\cB$ be a $G$-graded finite-dimensional unital Jordan algebra over $\RR$ that is graded-simple and has degree $2$ 
over its center $\LL$. Let $H$ be the support of $\LL$, $\bG=G/H$, and $\pi$ be the natural homomorphism $G\to\bG$. 
If $\LL_e=\RR$, pick a character $\chi:G\to U\subset\CC^\times$ such that $\LL\cong\lpc(\RR)$.
\begin{enumerate} 
\item[1.] If $\LL_e=\RR$ then $\cB$ is graded-isomorphic to $\lpc(\cJ(\kappa,\sigma))$ for a unique balanced function $\kappa:\bG\to\ZZ_{\ge 0}$ with finite support 
and a signature function $\sigma:\bG_{[2]}\to\ZZ$, which is determined up to multiplication by the characters $\bG_{[2]}\to\{\pm 1\}$ with trivial restriction to  
$\ker\tau$, where $\tau:\bG_{[2]}\to H/H^{[2]}$ is given by Equation~\eqref{eq:def_tau_again}.
\item[2.] If $\LL_e\cong\CC$ then $\cB$ is graded-isomorphic to $L_\pi(\cJ_\CC(\kappa))$ 
for a unique balanced function $\kappa:\bG\to\ZZ_{\ge 0}$ with finite support.\qed 
\end{enumerate}
\end{theorem}

The graded algebras in Theorem \ref{th:J_deg2_iso} can be obtained as Jordan algebras of symmetric bilinear forms over $\LL$. Indeed, consider the $G$-graded 
$\LL$-module $\cV=\lpc(V)$ and define a symmetric $\LL$-bilinear form $B:\cV\times\cV\to\LL$ as follows:
\[
B(v_1\ot u_{g_1},v_2\ot u_{g_2})\bydef b(v_1,v_2)\ot u_{g_1}u_{g_2}\text{ for all }v_1\in V_{\pi(g_1)},\,v_2\in V_{\pi(g_2)},\,g_1,g_2\in G.
\]  
Since $b(v_1,v_2)=0$ unless $g_1g_2\in H$, the right-hand side is indeed in $\LL$, and $B$ satisfies $B(\cV_{g_1},\cV_{g_2})\subset\LL_{g_1g_2}$.
Hence, we can define the $G$-graded Jordan algebra $\cJ(\cV,B)$, and the canonical isomorphism 
$(\RR 1\oplus V)\otimes\RR^\gamma G\cong \RR^\gamma G\oplus (V\otimes \RR^\gamma G)$ restricts to an isomorphism 
$\lpc(\cJ(V,b))\cong \cJ(\cV,B)$. 

Now choose $g_j\in G$ such that $\pi(g_j)=\bg_j$ for $1\le j\le m$ and $g_j'\in G$ such that $\pi(g_j')=\bg_j$ for $m<j\le l$. 
(If $\bg_j=\bar{e}$, we take $g_j=e$.)
Then $\cV$ can be identified with $V\otimes\LL$ as a graded $\LL$-module through the isomorphism sending, for each $1\le j\le m$, 
$v\ot u_h\mapsto v\ot u_{g_j}u_h$ for all $v\in V_{\bg_j}$, $h\in H$, and for each $m< j\le l$, $v\ot u_h\mapsto v\ot u_{g'_j}u_h$ for all $v\in V_{\bg'_j}$, $h\in H$,
and $v\ot u_h\mapsto v\ot u_{(g'_j)^{-1}}u_h$ for all $v\in V_{\bg''_j}$, $h\in H$. Then any choice of bases in $V_{\bg_j}$, $V_{\bg_j'}$ and $V_{\bg''_j}$
gives us an $\LL$-basis of $\cV$. We can choose the bases so that $B$ is represented by the matrix
\[
I_{p_1,q_1}\otimes\delta_1 u_{h_1}\oplus\cdots\oplus I_{p_m,q_m}\otimes\delta_m u_{h_m}\oplus
\begin{bmatrix}
0&I_{k_{m+1}}\\
I_{k_{m+1}}&0
\end{bmatrix}\otimes 1
\oplus\cdots\oplus
\begin{bmatrix}
0&I_{k_l}\\
I_{k_l}&0
\end{bmatrix}\otimes 1
\]
where $\delta_j=\gamma(\bg_j,\bg_j)\in\{\pm 1\}$ and $h_j=g_j^2$. 

If we consider these graded algebras up to equivalence rather than isomorphism, some further simplifications can be made. 
First of all, it is convenient to change the grading group as follows. 
Let $\wt{G}$ be the abelian group generated by $H$ and the new symbols 
$\tilde{g}_1,\ldots,\tilde{g}_m,\tilde{g}'_{m+1},\tilde{g}''_{m+1},\ldots,\tilde{g}'_{l},\tilde{g}''_{l}$ subject to the following relations: 
for $1\le j\le m$, we impose $\tilde{g}_j^2=h_j$ if $\bg_j\ne\bar{e}$ and $\tilde{g}_j=e$ otherwise, while for $m<j\le l$, we impose $\tilde{g}'_j\tilde{g}''_j=e$. 
Then $H$ is a subgroup of $\wt{G}$ and we have $\wt{G}/H\cong\ZZ_2^m\times\ZZ^{l-m}$ if $\bg_j\ne\bar{e}$ for all $j$ and $\wt{G}/H\cong\ZZ_2^{m-1}\times\ZZ^{l-m}$
otherwise. We define a $\wt{G}$-grading on $\cV$ by declaring, for each $1\le j\le m$, 
$\cV_{g_jh}$ to be the component of degree $\tilde{g}_jh$ for all $h\in H$, and similarly for $m<j\le l$. The $\LL$-bilinear form $B$ is compatible 
with this new grading, so we obtain a $\wt{G}$-grading on $\cJ(\cV,B)$, which is equivalent to the original $G$-grading.

\begin{remark}
We can say more: the $G$-grading on $\cJ(\cV,B)$ is induced from the $\wt{G}$-grading by the homomorphism $\alpha:\wt{G}\to G$ extending the identity on $H$
and sending $\tilde{g}_j\mapsto g_j$, $\tilde{g}'_j\mapsto g'_j$, and $\tilde{g}''_j\mapsto (g'_j)^{-1}$. 
It follows that $\wt{G}$ is the \emph{universal abelian group} (see e.g. \cite[\S 1.2]{EKmon}) of the grading on $\cJ(\cV,B)$.
\end{remark}

Now replace $G$ by $\wt{G}$ and drop the tilde. If $\LL$ is nonsplit, there exists $h\in H$ such that $u_h^2=-1$, and we can use it to adjust our  
$\LL$-basis of $\cV$ to make $\delta_j=1$ or $\delta_j=-1$ arbitrarily, but note that this also changes $g_j$, so we cannot apply this if $\bg_j=\bar{e}$. 
Therefore, we can make the matrix of $B$ look as follows:
\[
I_{p_1,q_1}\otimes u_{h_1}\oplus\cdots\oplus I_{p_m,q_m}\otimes u_{h_m}\oplus
\begin{bmatrix}
0&I_{k_{m+1}}\\
I_{k_{m+1}}&0
\end{bmatrix}\otimes 1
\oplus\cdots\oplus
\begin{bmatrix}
0&I_{k_l}\\
I_{k_l}&0
\end{bmatrix}\otimes 1
\]
where we may assume $p_j\ge q_j$ if $\bg_j\ne\bar{e}$.
If $\LL\cong\RR H$ then we may assume $u_h=h$, so $\delta_j=1$, but note that the isomorphism condition in 
Theorem~\ref{th:J_deg2_iso} allows us to interchange $p_j$ and $q_j$ arbitrarily if $\bg_j\notin\ker\tau$.  
Therefore, we can make the matrix of $B$ look as above, where we may assume $p_j\ge q_j$ if $\bg_j\notin\ker\tau$.

\subsection{Graded-division Jordan algebras of degree 2}

In a unital Jordan algebra $\cJ$, an inverse of an element $a$ is $b\in\cJ$ such that $ab=1$ and $a^2b=a$. The existence of such $b$ is equivalent to the 
invertibility of the operator $U_a\bydef 2L_a^2-L_{a^2}$. If this is the case, the element $b$ is unique and given by $b=U_a^{-1}(a)$ (see e.g. \cite[p.~52]{J}),
so the usual notation $a^{-1}$ is used. $\cJ$ is called a \emph{division algebra} if every nonzero element is invertible. 
It is clear that if $\cJ$ is $G$-graded and $a\in\cJ_g$ is invertible then $a^{-1}\in\cJ_{g^{-1}}$. 
As stated in the introduction, $\cJ$ is a \emph{graded-division algebra} if every nonzero homogeneous element is invertible. 
Note, however, that, for an invertible element $a$, the operator $L_a$ is not necessarily invertible, 
and the set of invertible elements need not be closed under multiplication.
 
The Jordan algebra of a symmetric bilinear form, $\cJ(V,b)$, is a division algebra if and only if $b(v,v)$ is not a square in $\FF$, for any $0\ne v\in V$ 
(see \cite[Exercise 2, p. 54]{J}). Clearly, in the case $\FF=\RR$, this is equivalent to $b$ being negative, i.e., $b(v,v)<0$, for any $0\ne v\in V$. 

\begin{proposition} 
Suppose we have a $G$-grading on the Jordan algebra $\cJ=\cJ(V,b)$ over any field $\FF$, $\chr{\FF}\ne 2$. 
Then $\cJ$  is a graded-division algebra if and only if the following conditions hold:
\begin{enumerate}
\item[(i)] all components $V_g$ are nonisotropic (i.e., the restriction of the quadratic form $b(v,v)$ to $V_g$ does not represent $0$);
\item[(ii)] for every $0\ne v\in V_e$, $b(v,v)$ is not a square in $\FF$. 
\end{enumerate}
Condition (i) implies that all elements of the support have order $\le 2$.
\end{proposition}

\begin{proof}
Since $\cJ_e=\FF 1\oplus V_e=\cJ(V_e,b_e)$, where $b_e$ is the restriction of $b$ to $V_e$, the nonzero elements of $\cJ_e$ are invertible if and only if 
condition (ii) holds, which also implies condition (i) for $g=e$. Now suppose $0\ne v\in \cJ_g$ for some $g\ne e$. Then $v\in V_g$ and hence $v^2=b(v,v)1$. 
If $b(v,v)=0$ then $v$ is not invertible because $v^2b=0\ne v$ for any $b\in\cJ$. 
Conversely, if $b(v,v)\ne 0$ then $v^{-1}=\frac{1}{b(v,v)}v$. Therefore, condition (i) 
is equivalent to the invertibility of all nonzero elements of $\cJ_g$. If $g^2\ne e$ then $V_g$ is totally isotropic.
%$U_v(\alpha 1+w)=2v(\alpha v+b(v,w)1)-b(v,v)(\alpha 1+w)=\alpha b(v,v)1+2b(v,w)v-b(v,v)w$  $=b(v,v)\id$.
\end{proof}

Recalling Remark \ref{rem:div} and applying the above proposition to our $\cJ(\kappa,\sigma)$ and $\cJ_\CC(\kappa)$, with $\bG$ playing the role of $G$, 
we obtain the following:

\begin{corollary}\label{cor:J_deg2_div}
Let $\cB$ be as in Theorem \ref{th:J_deg2_iso}. Then $\cB$ is a graded-division algebra if an only if one of the following holds:
\begin{enumerate}
\item[1.] $\LL_e=\RR$, the support of $\kappa$ is contained in $\bG_{[2]}$, $|\sigma(\bg)|=\kappa(\bg)$ for all $\bg\in\bG_{[2]}$, 
and $\sigma(\bar{e})=-\kappa(\bar{e})$;
\item[2.] $\LL_e\cong\CC$, the support of $\kappa$ is contained in $\bG_{[2]}$, $\kappa(\bg)\in\{0,1\}$ for all $\bg\in\bG_{[2]}$, 
and $\kappa(\bar{e})=0$.\qed 
\end{enumerate}
\end{corollary}

Note that in item 2, the conditions imply that all nonzero homogeneous components of $\cB$ are $1$-dimensional, whereas in item 1, they can have arbitrary dimensions.
This is in contrast with alternative (in particular, associative) graded-division algebras, where all nonzero components must have the same dimension 
(which, if finite, must be $1$, $2$, $4$ or $8$ over $\RR$).

\section*{Acknowledgments}
The authors are grateful to Alberto Elduque and Efim Zelmanov for fruitful discussions.

\end{document}